\documentclass[12pt]{amsart}

\usepackage{amsmath,amscd}
\usepackage{amssymb}
\usepackage{amsthm}

\usepackage{mathrsfs}
\usepackage{hyperref}
\usepackage{comment}

\usepackage{graphicx}
\usepackage{psfrag}

\usepackage{comment}
\usepackage{epstopdf}

\usepackage{calc}
\newenvironment{tehtratk}%
             {\begin{list}{\arabic{enumi}.}{\usecounter{enumi}%
              \setlength{\labelsep}{0.5em}%
              \settowidth{\labelwidth}{\arabic{enumi}.}%
              \setlength{\leftmargin}{\labelwidth+\labelsep}}}%
             {\end{list}}

\newtheorem{Theorem}{Theorem}[section]

\newtheorem{Lemma}[Theorem]{Lemma}
\newtheorem{Corollary}[Theorem]{Corollary}

\newtheorem{Proposition}[Theorem]{Proposition}

\theoremstyle{definition}

\newtheorem{Definition}[Theorem]{Definition}

\newtheorem{Remark}[Theorem]{Remark}

\newtheorem*{Remarknonum}{Remark}

\numberwithin{equation}{section}

\newcommand{\mR}{\mathbb{R}}                    
\newcommand{\abs}[1]{\lvert #1 \rvert}          
\newcommand{\norm}[1]{\lVert #1 \rVert}         

\newcommand{\mH}{\mathcal{H}}

\newcommand{\Int}{\mbox{Int}}

\newcommand{\R}{\mathbb{R}}

\newcommand{\p}{\partial}
\newcommand{\eps}{\varepsilon}

\begin{document}

\begin{abstract}
We consider a conformally invariant version of the Calder\'on problem, where the objective is to determine the conformal class of a Riemannian manifold with boundary from the Dirichlet-to-Neumann map for the conformal Laplacian. The main result states that a locally conformally real-analytic manifold in dimensions $\geq 3$ can be determined in this way, giving a positive answer to an earlier conjecture \cite[Conjecture 6.3]{LassasUhlmann}. The proof proceeds as in the standard Calder\'on problem on a real-analytic Riemannian manifold, but new features appear due to the conformal structure. In particular, we introduce a new coordinate system that replaces harmonic coordinates when determining the conformal class in a neighborhood of the boundary.
\end{abstract}

\title{The Calder{\'o}n problem for the conformal Laplacian}

\author[M. Lassas]{Matti Lassas}
\address{Department of Mathematics and Statistics, University of Helsinki}
\email{matti.lassas@helsinki.fi}

\author[T. Liimatainen]{Tony Liimatainen}
\address{Department of Mathematics and Statistics, University of Helsinki}
\email{tony.liimatainen@helsinki.fi}

\author[M. Salo]{Mikko Salo}
\address{Department of Mathematics and Statistics, University of Jyv\"askyl\"a}
\email{mikko.j.salo@jyu.fi}


\date{\today}


\maketitle



\section{Introduction}\label{sec_introduction}

\subsection{Calder{\'o}n problem} The anisotropic Calder\'on problem consists in determining a conductivity matrix of a medium, up to a change of coordinates fixing the boundary, from electrical voltage and current measurements on the boundary. In dimensions $\geq 3$ this problem may be written geometrically as the determination of a Riemannian metric on a compact manifold with boundary from Dirichlet and Neumann data of harmonic functions. More precisely, if $(M,g)$ is a compact oriented Riemannian manifold with smooth boundary, we consider the Dirichlet problem for the Laplace-Beltrami operator $\Delta_g$, 
\[
\Delta_g u = 0 \text{ in $M$,} \qquad u|_{\partial M} = f
\]
and define the Dirichlet-to-Neumann map (DN map) 
\[
\Lambda_g: C^{\infty}(\partial M) \to C^{\infty}(\partial M), \ \ \Lambda_g f = \partial_{\nu} u|_{\partial M}
\]
where $\partial_{\nu}$ is the normal derivative on $\partial M$. In dimensions $n \geq 3$, one has the coordinate invariance 
\[
\Lambda_g = \Lambda_{\phi^* g}
\]
for any diffeomorphism $\phi: M \to M$ fixing the boundary. If $n=2$, the Laplace-Beltrami operator is additionally conformally invariant, and one has 
\[
\Lambda_g = \Lambda_{c \phi^* g}
\]
whenever $\phi: M \to M$ is a diffeomorphism fixing the boundary, and $c \in C^{\infty}(M)$ is a positive function with $c|_{\partial M} = 1$ and $\partial_{\nu} c|_{\partial M} = 0$.

The geometric Calder\'on problem amounts to showing that the DN map $\Lambda_g$ determines the manifold $(M,g)$ modulo the above invariances. This has been verified in \cite{LassasUhlmann} in the following cases. If $(M,g)$ is a compact connected $C^{\infty}$ Riemannian manifold with $C^{\infty}$ boundary, then:

\begin{itemize}
\item[(a)] 
If $n=2$, the DN map $\Lambda_g$ determines the conformal class of $(M,g)$.
\item[(b)] 
If $n \geq 3$ and if $M$, $\partial M$ and $g$ are real-analytic, then the DN map $\Lambda_g$ determines $(M,g)$.
\end{itemize}

Related results are given in \cite{LeeUhlmann, LTU}, and an analogous result for Einstein manifolds (which are real-analytic in the interior) is proved in \cite{GuSaBar}. It remains a major open problem to remove the real-analyticity condition when $n \geq 3$; for recent progress in the case of conformally transversally anisotropic manifolds see \cite{DKSaU, DKLS}.

\subsection{Conformal invariance} Given the conformal invariance of the Calder\'on problem when $n=2$, \cite{LassasUhlmann} formulated an analogous inverse problem that is conformally invariant in any dimension and involves the conformal Laplacian. This inverse problem reduces to the usual Calder\'on problem when $n=2$, and it was conjectured that uniqueness holds for locally conformally real-analytic manifolds in dimensions $\geq 3$ (\cite[Conjecture 6.3]{LassasUhlmann}, 
see also \cite{U1,U2}). In this work we give a positive answer to this conjecture.

Inverse problems for the conformal Laplacian are closely related to counterexamples for inverse problems and invisibility cloaking. Invisibility cloaking means the possibility, both theoretical and practical, of shielding a region or object from detection via physical waves, see~\cite{AE,Ammari,GKLU,GKLU5,GKLU6,GLU2,GLU3,KSVW, MN,Le,PSS1}.
 
The first counterexamples for elliptic inverse problems using blow-up maps were developed for the Laplace-Beltrami operator on two-dimensional manifolds in~\cite{LTU} and they were based on the conformal invariance of harmonic functions. In~\cite{LTU} it was shown that if one removes one point from a compact manifold and blows up the metric using a conformal transformation, one obtains a non-compact manifold whose DN map coincides with that of a compact manifold. Thus the DN map for the Laplace-Beltrami operator does not determine even the topology of a non-compact manifold.

The blow-up of the metric near a point corresponds to making a hole, or a cavity, in the manifold. In 2003, other types of blow-up maps were used in~\cite{GLU3} to construct examples where objects were hidden inside holes created by blow-up maps. This led to counterexamples for Calder\'on's problem and to invisibility cloaking constructions for the conductivity equation. The interest in cloaking surged in 2006 when it was realized that practical cloaking constructions are possible using so-called metamaterials. The construction of Leonhardt
\cite{Le} was based on conformal mapping on a nontrivial Riemann surface. At the same time, Pendry et al
\cite{PSS1} proposed a cloaking construction for Maxwell's equations using a blow-up map and the idea was demonstrated in laboratory experiments~\cite{SMJCPSS}. For reviews on the topic, see~\cite{GKLU-BAMS,GKLU-SIAM}.

The conformal Laplacian is also called the Yamabe operator due to its role in the (nowadays solved) Yamabe problem on compact closed manifolds. The solution of the Yamabe problem is a positive function solving a nonlinear eigenvalue equation for the conformal Laplacian, see e.g.~\cite{LeeParker} and references therein. The Yamabe problem also has a counterpart on manifolds with boundary, which is settled in most cases~\cite{Escobar, Marques, Brendle}. Some other related works on the conformal Laplacian are~\cite{ParkerSteven, HJ}.

Finally, the conformal Laplacian and its higher order generalizations have appeared quite recently in physics in connection with the so-called AdS/CFT (Anti-de Sitter/Conformal Field Theory) 
correspondence in quantum gravity. (See e.g.~\cite{Witten, Anderson} for AdS/CFT correspondence.) The basic ingredients in this theory are Poincar\'e-Einstein metrics and manifolds, which are special manifolds with boundary, equipped with conformally compactified metrics~\cite{Anderson, Anderson2}. The geometric nature of the theory has generated significant interest in its mathematical aspects~\cite{FG, FGbook}. 
In the seminal work~\cite{GrahamZworski}, the scattering matrix of Poincar\'e-Einstein type manifolds is related to the conformal Laplacian and other ``conformally invariant powers of the Laplacian''. DN type maps for these operators are studied in~\cite{Gover, BransonGover}.

\subsection{Inverse problem for the conformal Laplacian}

Let $(M,g)$ be a compact connected oriented Riemannian manifold with smooth boundary, with $n = \dim(M)$. Consider the conformal Laplacian 
\[
L_g = -\Delta_g + \frac{n-2}{4(n-1)}S_g
\]
where $\Delta_g = -\delta d$ is the Laplace-Beltrami operator (with negative spectrum) and $S_g$ is the scalar curvature. The conformal Laplacian has the conformal scaling property  
\[
L_{cg} u = c^{-\frac{n+2}{4}} L_g (c^{\frac{n-2}{4}} u)
\]
for any smooth positive function $c$.

Assume that $0$ is not a Dirichlet eigenvalue for $L_g$ (this is the case for instance when $S_g \geq 0$). Then for any $f \in C^{\infty}(\partial M)$ the Dirichlet problem 
\[
L_g u = 0 \text{ in } M, \quad u|_{\partial M} = f,
\]
has a unique solution $u \in C^{\infty}(M)$ and we may define the Dirichlet-to-Neumann map (DN map) related to $L_g$ by 
\[
N_g: C^{\infty}(\partial M) \to C^{\infty}(\partial M), \ \ N_g f = \partial_{\nu} u|_{\partial M}.
\]
It follows that if $c \in C^{\infty}(M)$ is a positive function satisfying $c|_{\partial M} = 1$ and $\partial_{\nu} c|_{\partial M} = 0$, then 
\[
N_{cg} = N_g.
\]
Consequently, one can only expect to determine $g$ up to such a conformal transformation from $N_g$.

It was conjectured in  \cite[Conjecture 6.3]{LassasUhlmann}, 
see also \cite{U1,U2},
 that if $(M,g)$ is a manifold of dimension $n \geq 3$ that is locally conformal to a real analytic manifold, and if $0$ is not a Dirichlet eigenvalue of $L_g$, then $N_g$ determines a manifold conformal to $(M,g)$. Generally speaking, a Riemannian manifold is locally conformal to a real analytic manifold if near each point there is a coordinate chart so that the corresponding coordinate representation of the metric is a real analytic matrix field up to a $C^\infty$ conformal factor. We will make this precise in Definition~\ref{locally_conformally_ra} below.

We give a positive answer to this conjecture in the following sense.

\begin{Theorem}\label{mainthm}
 Let $(M,g)$ be a compact locally conformally real analytic Riemannian manifold with boundary, $n=\dim(M)\geq 3$. Assume that $0$ is not a Dirichlet eigenvalue of the conformal Laplacian. Then the DN map defines the manifold $(M,g)$ up to a conformal scaling $c$ and a diffeomorphism $F$, which satisfy the following conditions:
\begin{eqnarray}\label{eq: gauge fix}
 c|_{\partial M}=1, \quad \partial_\nu c|_{\partial M}=0 \mbox { and } F|_{\partial M}=\mathrm{Id}.
\end{eqnarray}
\end{Theorem}

We proceed to define locally conformally real analytic manifolds.

\subsection{Locally conformally real analytic manifolds}
A Riemannian manifold is said to be locally conformally real analytic if around each of its points there are local coordinates where the coordinate representation is real analytic up to a conformal scaling. If the point is a boundary point, we assume (in this paper) boundary conditions for the scaling. Precisely we define:

\begin{Definition}\label{locally_conformally_ra}
\begin{tehtratk}
\item[(a)]
A Riemannian manifold $(M,g)$ \emph{without boundary} is locally conformally real analytic if near every point $p\in M$ there are local coordinates $\phi:U\to \R^n$, $\phi(x)=(x^i(x))_{i=1}^n$, such that in these coordinates $g$ has the form 
\begin{equation}\label{cra}
g_{ij}(x)=s(x)\, h_{ij}(x),
\end{equation}
with $s\in C^\infty(\Omega)$ and $h_{ij}\in C^\omega(\Omega)$ real analytic, $\Omega = \phi(U) \subset \R^n$ open and connected.  

\item[(b)] 
$\!\!$ A Riemannian manifold $(M,g)$ \emph{with boundary} $\p M$ is locally conformally real analytic if interior points satisfy the condition in (a), and if for every boundary point $p\in \p M$ there is a boundary chart $\phi:U\to \R^n$, with coordinates $(x',x^n)$, $x'\in \R^{n-1}$, $x^n\geq 0$, such that
\[
g_{ij}(x) = s(x) h_{ij}(x)
\]
where $s\in C^\infty(\overline{\Omega})$ and $h_{ij}\in C^\omega(\overline{\Omega})$, $\Omega = \phi(U) \subset \{ x_n \geq 0 \}$, and additionally 
\[
s(x',0) \in C^\omega(\Gamma)
\]
where $\Gamma:=\Omega\cap\{x^n=0\}\subset \R^{n-1}$.
\end{tehtratk}
\end{Definition}

The notation $h_{ij}\in C^\omega(\overline{\Omega})$ means that $h_{ij}$ is real analytic up to the boundary and thus has a converging Taylor expansion with positive radius of convergence also at points of $\Gamma$. Also, by a boundary chart near a point $p\in \p M$ we mean a coordinate chart $\phi:U\to \mathbb{H}^n=\{(x',x_n)\in \R^n: x'\in \R^{n-1}, \ x^n\geq 0\}$ satisfying $\phi(U\cap \p M)\subset \{x^n=0\}$. Boundary normal coordinates provide an example of a boundary chart.

We remark that we do not assume in the definition of a locally conformally real analytic manifold that the manifold has a real analytic atlas. However, we prove in Appendix \ref{section_realanalytic_structure} (Proposition~\ref{Comeg_str}) that this is true at least if the manifold has no boundary.

Examples of locally conformally real analytic manifolds (without boundary) are those conformal to Einstein manifolds, but also Bach and obstruction flat manifolds. Bach flat manifolds are conformally invariant generalizations of Einstein manifolds in dimension $4$, while obstruction flat manifolds serve a similar purpose in even dimensions $n>4$, see e.g.~\cite{FG, FGbook}. That Bach and obstruction flat manifolds are locally conformally real analytic can be seen for example by using $n$-harmonic coordinates, so that the determinant normalized metric satisfies an elliptic PDE in these coordinates~\cite{LS2}.

\subsection{Outline of proof}

The proof of the main result may be divided in three parts.

\begin{tehtratk}

\item[1.] 
We will begin by determining the Taylor series of the metric \emph{on the boundary} following \cite{DKSaU}. To do this, one needs to find a suitable conformal scaling and determine the Taylor series in boundary normal coordinates for the scaled metric. One can think of this as a gauge fixing process for the conformal and diffeomorphism invariances of the DN map. 

The boundary normal coordinates, where we determine the jet of the (conformal) metric on the boundary, might not be real analytic. This causes a difficulty in the next step where we determine the metric near the boundary from its Taylor expansion.

\item[2.] 
Next we determine the metric \emph{near the boundary}. For this we introduce a new real analytic coordinate system associated with the conformal Laplacian. We call the new coordinate system \emph{$Z$-coordinates}. In these coordinates we can employ the assumption that the manifolds are locally conformally real analytic. This technique is in part motivated by the work~\cite{GuSaBar} on inverse problems on Einstein manifolds.

\item[3.] 
After determining the metric near the boundary, we follow the embedding argument of ~\cite{LTU}. There the authors give a new proof for the result of~\cite{LassasUhlmann} stating that on real analytic manifolds ($n\geq 3$) the DN map associated to the Laplace-Beltrami operator determines the manifold up to a boundary preserving diffeomorphism. We extend the methods of~\cite{LTU} to the conformally invariant setting of this paper.

\end{tehtratk}

Finally, we remark that just knowing a real analytic metric near the boundary does not determine the Riemannian manifold itself, as can be seen from the following example.
\medskip

\noindent
{\bf Example.} (Real analytic manifolds that are isometric near the boundary)
Let $N_1$ be the $n$-dimensional sphere $\mathbb S^n\subset \R^{n+1}$, $n\geq 3$, and let $p_0\in \mathbb S^n$
be the North Pole. Also, let $N_2$ be the $n$-dimensional projective space $\R\mathbb{P}^n$,
constructed by taking the closed upper half-sphere $\mathbb S^n\cap (\R^{n}\times [0,\infty))$ and identifying
the antipodal points on the boundary $\mathbb S^{n-1}\times \{0\}\subset  \R^{n+1}$. 

This well known construction defines a 2-to-1 map $H:N_1=\mathbb S^n\to N_2= \R\mathbb{P}^n$. We use on $N_1$ the metric $g_1$ that is inherited from 
$\R^{n+1}$ and on $N_2$ the metric $g_2=H_*g_1$. Also, let
$B_1=B_{N_1}(p_0,\frac 12)$ be the ball of $N_1$ having center at the North Pole $p_0$ and
radius $1/2$. Let $B_2=H(B_1)$. Finally, define $M_1=N_1\setminus B_1$ and  $M_2=N_2\setminus B_2$.
Then $M_1$ and $M_2$ are real-analytic manifolds with boundary such that the $\tau$-neighborhoods
of the boundaries with $\tau<\pi/2-1$ are isometric. In particular, 
the $C^\infty$  jets of the metric tensors in their boundary normal coordinates on these manifolds coincide. However, $M_1$ and $M_2$ are not isometric, nor even homeomorphic ($M_1$ is simply connected, but $M_2$ is not since its covering space $\mathbb S^n \setminus (B(p_0,1) \cup B(-p_0,1))$ has two boundary components whereas $M_2$ has only one).

We will see below that the symbol of the DN map (considered as a pseudo-differential operator) of the conformal Laplacian on $(M_j,g_j)$, $j=1,2$, determines the  $C^\infty$ jet of a metric conformal to $g_j$ in boundary normal coordinates of the conformal metric.

However, by Theorem \ref{mainthm} the DN map of the conformal Laplacian determines  the manifold $(M,g)$ up to a conformal transformation satisfying (\ref{eq: gauge fix}). 
Hence, the above elementary example of manifolds $M_1$ and $M_2$ shows that symbol (modulo smoothing symbols in class $S^{-\infty}$) of the DN map is not sufficient to determine uniquely the conformal type of the real-analytic manifold. Rather, some additional 
global properties of the DN map are needed.

\subsection*{Acknowledgements}

The authors are grateful to Plamen Stefanov, who pointed out a mistake in the boundary determination result in \cite[Section 8]{DKSaU} and provided a corrected proof. All authors were supported by the Academy of Finland (Centre of Excellence in Inverse Problems Research). T.L.\ and M.S.\ were also partly supported by an ERC Starting Grant (grant agreement no 307023).

\section{Boundary determination}
The first step is the following boundary determination result. It shows that the DN map $N_g$ of the conformal Laplacian determines the Taylor series of some conformal metric, in boundary normal coordinates of the new metric. 

\begin{Lemma} \label{lemma_boundarydetermination}
Let $(M,g)$ be a compact oriented Riemannian manifold with smooth boundary, with $n = \dim(M) \geq 3$.
\begin{enumerate}
\item[(a)]
Let $U \subset M$ be open and assume that $\Gamma = U \cap \partial M$ is nonempty. If $c \in C^{\infty}(U)$ is a positive function satisfying 
\begin{equation} \label{c_conditions_gamma}
c|_{\Gamma} = 1, \qquad \partial_{\nu} c|_{\Gamma} = 0, \qquad \mathcal{L}_{\tilde{N}}^j \tilde{H} |_{\Gamma} = 0 \ \ (j \geq 1),
\end{equation}
where $\tilde{g} = cg$ (defined in $U$), then the knowledge of $N_g f|_{\Gamma}$ for any $f \in C^{\infty}_c(\Gamma)$ determines 
\[
\mathcal{L}_{\tilde{N}}^j \tilde{g}|_{\Gamma}, \qquad j \geq 0.
\]
\item[(b)]
There exists a positive function $c \in C^{\infty}(M)$ satisfying \eqref{c_conditions_gamma} on $\partial M$.
\end{enumerate}
Here $\mathcal{L}_{\tilde{N}}$ is the Lie derivative of the $2$-tensor $\tilde{g}$, $\tilde{N}$ is the extension of the $\tilde{g}$-inner unit normal vector of $\partial M$ which is parallel along $\tilde{g}$-normal geodesics, and $\tilde{H} = \Delta_{\tilde{g}} (d_{\tilde{g}}(\,\cdot\,,\partial M))$ is the mean curvature of the hypersurfaces with fixed $\tilde{g}$-distance from $\partial M$.
\end{Lemma}
\begin{Remarknonum}
If $p \in \Gamma$ and if $(\tilde{x}',\tilde{x}_n)$ are $\tilde{g}$-boundary normal coordinates in some neighbourhood of $p$ in $M$, so that $\tilde{g} = \tilde{g}_{\alpha \beta} \,d\tilde{x}^{\alpha} \,d\tilde{x}^{\beta} + (d\tilde{x}^n)^2$ where $\alpha$ and $\beta$ are summed from $1$ to $n-1$, then 
\[
\mathcal{L}_{\tilde{N}}^j \tilde{g}|_{\Gamma} = \partial_{\tilde{x}_n}^j \tilde{g}_{\alpha \beta}(\tilde{x}',0) \,d\tilde{x}^{\alpha} \,d\tilde{x}^{\beta} \qquad (j \geq 1),
\]
so knowing $N_g$ on $\Gamma$ in fact determines $\partial_{\tilde{x}_n}^j \tilde{g}_{\alpha \beta}|_{\Gamma}$ for all $j \geq 0$.
\end{Remarknonum}

Combining (a) and (b) above, we see that the knowledge of $N_g f|_{\Gamma}$ for any $f \in C^{\infty}_c(\Gamma)$ determines the Taylor series of some conformal metric $\tilde{g}$ in $\tilde{g}$-boundary normal coordinates on $\Gamma$. Lemma~\ref{lemma_boundarydetermination} implies the following kind of result.

\begin{Lemma} \label{lemma_boundarydetermination_realanalytic}
Let $(M,g)$ be a compact oriented Riemannian manifold with smooth boundary, with $n = \dim(M) \geq 3$. Let $p \in \partial M$, and assume that for some neighborhood $U$ of $p$ in $M$ there is a positive function $c \in C^{\infty}(U)$ satisfying \eqref{c_conditions_gamma} for $\Gamma = U \cap \partial M$.
Then the knowledge of $N_g f|_{\Gamma}$ for any $f \in C^{\infty}_c(\Gamma)$ determines 
\[
\partial_{\tilde{x}_n}^j \tilde{g}_{\alpha \beta}|_{\Gamma} \qquad (j \geq 0)
\]
in $\tilde{g}$-boundary normal coordinates at $p$. 
\end{Lemma}

Lemma \ref{lemma_boundarydetermination} (a) is essentially proved in \cite[Section 8]{DKSaU}. The first step there was a conformal normalization, replacing the metric $g$ by a conformal metric $\tilde{g}$ such that $\log \det(\tilde{g})$ has suitable Taylor series in $\tilde{g}$-boundary normal coordinates at the boundary. However, there is a mistake in this part of \cite[Section 8]{DKSaU}. The authors are grateful to Plamen Stefanov, who pointed out the mistake and provided a corrected proof for the conformal normalization statement (a related argument is given in \cite{StefanovYang}). Lemma \ref{lemma_boundarydetermination} (b) gives an invariant formulation of the conformal normalization condition in terms of the mean curvature $\tilde{H}$. The proof is partly based on the argument of Stefanov.

We will use the following notation. If $(M,g)$ is a compact oriented manifold with smooth boundary, let $r(x) = d_g(x,\partial M)$ so that $r$ is a smooth function near $\partial M$ and $\Gamma_t = \{Êx \in M \,;\, r(x) = t \}$ are smooth submanifolds for $t$ small with $\Gamma_0 = \partial M$. The vector field 
\[
N = \mathrm{grad}_g(r)
\]
is a unit normal vector field of $\Gamma_t$ for $t$ small. If $\nabla$ is the Levi-Civita connection of $(M,g)$, let $S$ be the $(1,1)$-tensor field near $\partial M$, 
\[
S(X) = \nabla_X N,
\]
so that $S|_{\Gamma_t}$ is the shape operator of $\Gamma_t$ and $S(N) = 0$. The $2$-tensor $h$ obtained from $S$ by lowering an index is given by  
\[
h = \mathrm{Hess}(r), \qquad h(X,Y) = \langle \nabla_X N, Y \rangle,
\]
and this corresponds to the scalar second fundamental form of $\Gamma_t$. Finally let $H$ be the mean curvature of the surfaces $\Gamma_t$, 
\[
H = \mathrm{Tr}_g(h) = \Delta_g r
\]
where $\Delta_g = -\delta_g d$ is the (negative) Laplace-Beltrami operator. (We omit the factor $\frac{1}{n-1}$ usually included in the definition of $H$.)

Let $(x',x_n)$ be any boundary normal coordinates, so that 
\[
g = g_{\alpha \beta} \,dx^{\alpha} dx^{\beta} + (dx^n)^2.
\]
We use the Einstein summation convention so that a repeated Greek index in upper and lower position is summed from $1$ to $n-1$, whereas Roman indices are summed from $1$ to $n$. Then one has 
\[
r = x_n, \qquad N = \partial_n,
\]
the scalar second fundamental form is given by 
\begin{align*}
h(\partial_{\alpha}, \partial_{\beta}) &= \langle \nabla_{\partial_{\alpha}} \partial_n, \partial_{\beta}Ê\rangle = \Gamma_{\alpha n}^l g_{l\beta} = \frac{1}{2} g^{lm}(\partial_{\alpha} g_{nm} + \partial_n g_{\alpha m} - \partial_m g_{\alpha n}) g_{l\beta} \\
 &= \frac{1}{2} \partial_n g_{\alpha \beta},
\end{align*}
and the mean curvature is given by 
\[
H = \frac{1}{2} g^{\alpha \beta} \partial_n g_{\alpha \beta}.
\]

If $(M,g)$ is given, we consider the conformal metric $\tilde{g} = e^{2\mu} g$ where $\mu \in C^{\infty}(M)$ is real valued. Denote by $\tilde{r}, \tilde{N}, \tilde{S}, \tilde{h}, \tilde{H}$ the above quantities in $(M,\tilde{g})$. We first compare the normal derivatives $\partial_n$ and $\tilde{\partial}_n$ with respect to $g$ and $\tilde{g}$, respectively.

\begin{Lemma} \label{lemma_normalderivative_comparison}
Let $(M,g)$ be a compact manifold with smooth boundary, let $\mu \in C^{\infty}(M)$ satisfy $\mu|_{\partial M} = \partial_{\nu} \mu|_{\partial M} = 0$, and let $\tilde{g} = e^{2\mu} g$. If $f$ is a smooth function in $M$, then 
\begin{align*}
\tilde{\partial}_n f|_{\partial M} &= \partial_n f|_{\partial M}, \\
\tilde{\partial}_n^2 f|_{\partial M} &= \partial_n^2 f|_{\partial M}, \\
\tilde{\partial}_n^3 f|_{\partial M} &= \partial_n^3 f - (\partial_n^2 \mu) \partial_n f|_{\partial M}
\end{align*}
and for any $m \geq 4$ 
\[
\tilde{\partial}_n^m f|_{\partial M} = \partial_n^m f - (\partial_n^{m-1} \mu) \partial_n f + \sum_{j=0}^{m-1} T^m_j(\partial_n^j f)|_{\partial M}
\]
where each $T^m_j$ is a tangential differential operator on $\partial M$ depending on $\mu$ only through $\mu|_{\partial M}, \partial_n \mu|_{\partial M}, \ldots \partial_n^{m-2} \mu|_{\partial M}$ and their tangential derivatives.
\end{Lemma}
\begin{proof}
Let $(x',x^n)$ be $g$-boundary normal coordinates near a boundary point $p$, and let $\eta(s)$ be any smooth curve through $p$. Write $\eta^j(s) = x^j(\eta(s))$. We claim that for any $m \geq 3$, one has 
\begin{multline}
\partial_s^m (f(\eta(s)) = \partial_{j_1 \cdots j_m} f(\eta) \dot{\eta}^{j_1} \cdots \dot{\eta}^{j_m} + \partial_j f(\eta) \partial_s^{m-1} \dot{\eta}^j \\
 + \sum_{l=2}^{m-1} \partial_{j_1 \cdots j_l}  f(\eta) \left[ \sum_{i_1 + \cdots +i_l = m-l} a^m_{i_1 \cdots i_l} (\partial_s^{i_1} \dot{\eta}^{j_1}) \cdots (\partial_s^{i_l} \dot{\eta}^{j_l}) \right] \label{f_normalderivative_first}
\end{multline}
where each $a^m_{i_1 \cdots i_l}$ is an absolute constant. In fact, one has 
\begin{align*}
\partial_s (f(\eta(s))) &= \partial_j f(\eta) \dot{\eta}^j, \\
\partial_s^2 (f(\eta(s))) &= \partial_{j k} f(\eta) \dot{\eta}^j \dot{\eta}^k + \partial_j f(\eta) \ddot{\eta}^j, \\
\partial_s^3 (f(\eta(s))) &= \partial_{j k l} f(\eta) \dot{\eta}^j \dot{\eta}^k \dot{\eta}^l + 3 \partial_{j k} f(\eta) \ddot{\eta}^j \dot{\eta}^k + \partial_j f(\eta) \dddot{\eta}^j
\end{align*}
which proves the case $m=3$. The claim follows by induction.

Now choose $\eta(s)$ to be the normal $\tilde{g}$-unit speed $\tilde{g}$-geodesic through $p$. Then $\eta(0) = p$, $\dot{\eta}(0) = \tilde{\partial}_n|_p$, and representing $\eta$ in terms of the $g$-boundary normal coordinates $(x',x^n)$ yields 
\[
\ddot{\eta}^l(s) = -\tilde{\Gamma}_{jk}^l(\eta(s)) \dot{\eta}^j(s) \dot{\eta}^k(s)
\]
by the $\tilde{g}$-geodesic equation.
The $\tilde{g}$-Christoffel symbols $\tilde{\Gamma}_{jk}^l$ are related to $g$-Christoffel symbols $\Gamma_{jk}^l$ by 
\[
\tilde{\Gamma}_{jk}^l = \Gamma_{jk}^l + (\partial_j \mu) \delta_k^l + (\partial_k \mu) \delta_j^l - g^{lq} (\partial_q \mu) g_{jk}.
\]
This implies that 
\[
\partial_s \dot{\eta}^l(s) = -\Gamma_{jk}^l \dot{\eta}^j \dot{\eta}^k - 2 (\partial_s \mu) \dot{\eta}^l + e^{-2\mu} g^{lq} \partial_q \mu|_{\eta(s)}
\]
where $\partial_s \mu = \partial_s (\mu(\eta(s)))$. Observe that one has $\Gamma_{nn}^l = 0$ for any $l$ since $(x',x^n)$ are $g$-boundary normal coordinates. This gives 
\begin{equation} \label{partialeta_general}
\partial_s \dot{\eta}^l(s) = -\Gamma_{\alpha \beta}^l \dot{\eta}^{\alpha} \dot{\eta}^{\beta} - 2\Gamma_{\alpha n}^l \dot{\eta}^{\alpha} \dot{\eta}^{n} - 2 (\partial_s \mu) \dot{\eta}^l + e^{-2\mu} g^{lq} \partial_q \mu|_{\eta(s)}.
\end{equation}
By induction, one sees that for $m \geq 1$ 
\begin{equation} \label{partialeta_general2}
\partial_s^m \dot{\eta}^l(s) = -2 (\partial_s^m \mu) \dot{\eta}^l + e^{-2\mu} g^{lq} \partial_s^{m-1} (\partial_q \mu(\eta(s))) + f^{ml}|_{\eta(s)}
\end{equation}
where each $f^{ml}|_{\eta(s)}$ depends on $\partial_{j_1 \cdots j_r} \mu(\eta(s))$ and $\partial_s^r \dot{\eta}^j(s)$ for $r \leq m-1$, and on $\partial_{j_1 \cdots j_r} g_{jk}$ for $r \leq m$. Here we also used \eqref{f_normalderivative_first}.

Restricting to $\partial M$, the condition $\mu|_{\partial M} = 0$ implies that $\dot{\eta}(0) = \tilde{N}(p) = N(p) = \partial_n|_p$ and thus 
\[
\dot{\eta}^l|_{\partial M} = \delta^l_n.
\]
Now \eqref{partialeta_general} and the condition $\partial_n \mu|_{\partial M} = 0$ give 
\[
\ddot{\eta}^l|_{\partial M} = 0.
\]
Taking $\partial_s$ of \eqref{partialeta_general}, evaluating on $\partial M$ and using the previous conditions gives 
\begin{align*}
\dddot{\eta}^l|_{\partial M} &= -2(\partial_s^2 \mu) \dot{\eta}^l + g^{lq} \partial_{nq} \mu|_{\partial M} \\
 &= -(\partial_n^2 \mu) \delta^l_n.
\end{align*}
Using \eqref{partialeta_general2} and \eqref{f_normalderivative_first}, it follows by induction that 
\begin{equation} \label{partialsm_eta_boundary}
\partial_s^m \dot{\eta}^l|_{\partial M} =  -2(\partial_n^m \mu) \delta^l_n + g^{lq} \partial_{q} \partial_n^{m-1} \mu + f^{ml}|_{\partial M}
\end{equation}
for some new functions $f^{ml}$ depending on $\partial_{j_1 \cdots j_r} \mu|_{\partial M}$ for $r \leq m-1$ and on $\partial_{j_1 \cdots j_r} g_{jk}|_{\partial M}$ for $r \leq m$. Since $\tilde{\partial}_n f(\eta(s)) = \partial_s (f(\eta(s)))$, one has $\tilde{\partial}_n^m f(\eta(s)) = \partial_s^m (f(\eta(s)))$. Thus inserting \eqref{partialsm_eta_boundary} in \eqref{f_normalderivative_first} gives that 
\begin{multline*}
\tilde{\partial}_n^m f|_{\partial M} = \partial_n^m f - (\partial_n f) \partial_n^{m-1} \mu + g^{\alpha \beta} (\partial_{\alpha} f) \partial_{\beta} \partial_n^{m-2} \mu + \sum_{j=0}^{m-1} T^m_j(\partial_n^j f)|_{\partial M}
\end{multline*}
where each $T^m_j$ is a tangential differential operator on $\partial M$ with coefficients depending on $\mu$ and its derivatives up to order $m-2$, and on $g_{pq}$ and their derivatives.
\end{proof}

\begin{Lemma} \label{lemma_tilde_boundary}
Let $(M,g)$ be a compact manifold with smooth boundary, let $\mu \in C^{\infty}(M)$ satisfy $\mu|_{\partial M} = \partial_{\nu} \mu|_{\partial M} = 0$, and let $\tilde{g} = e^{2\mu} g$. If $\tilde{r} = d_{\tilde{g}}(\,\cdot\,,\partial M)$ and if $(x',x_n)$ are $g$-boundary normal coordinates, then 
\begin{align*}
\tilde{r}|_{\partial M} &= 0, \\
\partial_n \tilde{r}|_{\partial M} &= 1, \\
\partial_n^2 \tilde{r}|_{\partial M} &=0, \\
\partial_n^3 \tilde{r}|_{\partial M} &=\partial_n^2 \mu|_{\partial M}, \\
\partial_n^m \tilde{r}|_{\partial M} &=  \partial_n^{m-1} \mu|_{\partial M} + f^m \quad (m \geq 4)
\end{align*}
where $f^m$ depends on $\mu|_{\partial M}$, $\partial_n \mu|_{\partial M}$, $\ldots, \partial_n^{m-2} \mu|_{\partial M}$ and their tangential derivatives.
\end{Lemma}
\begin{proof}
Clearly $\tilde{r}|_{\partial M} = 0$. Note that if $(\tilde{x}',\tilde{x}_n)$ are $\tilde{g}$-boundary normal coordinates, then $\tilde{r} = \tilde{x}_n$ and one has $\tilde{\partial}_n \tilde{r} = 1$ and $\tilde{\partial}_n^m \tilde{r} = 0$ for $m \geq 2$. We use Lemma \ref{lemma_normalderivative_comparison} to obtain 
\begin{align*}
\partial_n \tilde{r}|_{\partial M} &= \tilde{\partial}_n \tilde{r}|_{\partial M} = 1, \\
\partial_n^2 \tilde{r}|_{\partial M} &= \tilde{\partial}_n^2 \tilde{r}|_{\partial M} = 0, \\
\partial_n^3 \tilde{r}|_{\partial M} &= \tilde{\partial}_n^3 \tilde{r} + (\partial_n^2 \mu) \partial_n \tilde{r}|_{\partial M} = \partial_n^2 \mu|_{\partial M}.
\end{align*}
Finally, if $m \geq 4$ we obtain from Lemma \ref{lemma_normalderivative_comparison} that  
\begin{align*}
0 &= \tilde{\partial}_n^m \tilde{r}|_{\partial M} \\
 &= \partial_n^m \tilde{r} - (\partial_n^{m-1} \mu) \partial_n \tilde{r} + \sum_{j=0}^{m-1} T^m_j(\partial_n^j \tilde{r})|_{\partial M}
\end{align*}
where $T^m_j$ are tangential differential operators depending on $\mu|_{\partial M}$, $\partial_n \mu|_{\partial M}$, $\ldots, \partial_n^{m-2} \mu|_{\partial M}$ and their tangential derivatives. Inductively, we see that for $m \geq 3$ 
\[
\partial_n^m \tilde{r}|_{\partial M} = \partial_n^{m-1} \mu|_{\partial M} + f^m
\]
where $f^m$ depends on $\mu|_{\partial M}$, $\partial_n \mu|_{\partial M}$, $\ldots, \partial_n^{m-2} \mu|_{\partial M}$ and their tangential derivatives.
\end{proof}

\begin{proof}[Proof of Lemma \ref{lemma_boundarydetermination} (b)]
We look for $c$ in the form $c = e^{2\mu}$ for some $\mu \in C^{\infty}(M)$ with 
\begin{equation} \label{mu_conditions_initial}
\mu|_{\partial M} = \partial_{\nu} \mu|_{\partial M} = 0.
\end{equation}
If $\tilde{g} = cg$ and $\tilde{r} = d_{\tilde{g}}(\,\cdot\,,\partial M)$, we have 
\[
\tilde{H} = \Delta_{\tilde{g}} \tilde{r}.
\]
If $f$ is any smooth function, a computation shows that 
\[
\Delta_{\tilde{g}} f = e^{-2\mu} (\Delta_g f + (n-2) \langle d\mu, df \rangle_g).
\]
Let now $(x',x_n)$ be $g$-boundary normal coordinates, so that in these coordinates $g = g_{\alpha \beta} \,dx^{\alpha} \,dx^{\beta} + (dx^n)^2$. Then $\Delta_g = g^{jk} (\partial_{jk} - \Gamma_j \partial_k)$ where $\Gamma_j = g_{jl} \Gamma^l$ and $\Gamma^l = g^{jk} \Gamma_{jk}^l = -\Delta_g x^l$. In particular $\Gamma_n = -H$. It follows that 
\begin{align}
\tilde{H} &= e^{-2\mu} (\Delta_g \tilde{r} + (n-2) \langle d\mu, d\tilde{r} \rangle_g) \notag \\
 &= e^{-2\mu} \Big[ \partial_n^2 \tilde{r} + [ (n-2) \partial_n \mu + H ] \partial_n \tilde{r} \notag \\
 &\qquad \qquad \qquad + g^{\alpha \beta }(\partial_{\alpha \beta} \tilde{r} + [ (n-2) \partial_{\alpha} \mu - \Gamma_{\alpha} ] \partial_{\beta} \tilde{r}) \Big]. \label{tildeh_gbnc_expression}
\end{align}

Using Lemma \ref{lemma_tilde_boundary} and \eqref{mu_conditions_initial}, we have 
\[
\tilde{H}|_{\partial M} = H|_{\partial M}.
\]
For the first normal derivative, we use Lemmas \ref{lemma_normalderivative_comparison}--\ref{lemma_tilde_boundary} and \eqref{tildeh_gbnc_expression} and compute 
\begin{align*}
\tilde{\partial}_n \tilde{H}|_{\partial M} &= \partial_n \tilde{H}|_{\partial M} \\
 &= \partial_n^3 \tilde{r} + (n-2)\partial_n^2 \mu + \partial_n H|_{\partial M} \\
 &= (n-1) \partial_n^2 \mu + \partial_n H|_{\partial M}.
\end{align*}
Thus choosing $\mu$ so that 
\[
\partial_n^2 \mu|_{\partial M} = -\frac{1}{n-1} \partial_n H|_{\partial M}
\]
will ensure that $\tilde{\partial}_n \tilde{H}|_{\partial M} = 0$. If $m \geq 2$, we first use Lemma \ref{lemma_normalderivative_comparison} and the fact that $\partial_n \tilde{H}|_{\partial M} = 0$ to obtain that 
\[
\tilde{\partial}_n^m \tilde{H}|_{\partial M} = \partial_n^m \tilde{H} + \sum_{j=0}^{m-1} T_j^m(\partial_n^j \tilde{H})|_{\partial M}
\]
where $T_j^m$ are tangential differential operators depending on tangential derivatives of $\partial_n^r \mu|_{\partial M}$ for $r \leq m-2$. We then differentiate \eqref{tildeh_gbnc_expression}, which gives 
\[
\partial_n^m \tilde{H}|_{\partial M} = \partial_n^{m+2}Ê\tilde{r} + (n-2) \partial_n^{m+1} \mu + \partial_n^m H + f^m|_{\partial M}
\]
where $f^m$ depends on tangential derivatives of $\partial_n^r \mu|_{\partial M}$ for $r \leq m$, tangential derivatives of $\partial_n^r \tilde{r}|_{\partial M}$ for $r \leq m+1$, and on $g_{jk}$ and their derivatives. Combining these facts with Lemma \ref{lemma_tilde_boundary} yields that 
\[
\tilde{\partial}_n^m \tilde{H}|_{\partial M} = (n-1)\partial_n^{m+1} \mu + \partial_n^m H + f^m|_{\partial M}
\]
where $f^m$ depends on tangential derivatives of $\partial_n^r \mu|_{\partial M}$ for $r \leq m$ and on $g_{jk}$ and their derivatives. Thus we may choose $\partial_n^j \mu|_{\partial M}$ inductively for all $j \geq 2$ and apply Borel summation to get a smooth function $\mu$ near $p$ so that $\mu|_{\partial M} = \partial_{\nu} \mu|_{\partial M} = 0$ and $\tilde{\partial}_n^m \tilde{H}|_{\partial M} = 0$ near $p$ for all $m \geq 1$. Covering $\partial M$ with finitely many coordinate charts, doing this construction in each coordinate chart and applying a suitable partition of unity gives a function $\mu \in C^{\infty}(M)$ so that $c = e^{2\mu}$ has the required properties.
\end{proof}

\begin{proof}[Proof of Lemma \ref{lemma_boundarydetermination} (a)]
We use the results of \cite[Section 8]{DKSaU}. The main point is that in the notation of \cite{DKSaU} we have 
\[
N_g = \Lambda_{g,0,q_g}
\]
where $q_g = \frac{n-2}{4(n-1)} S_g$.

Let $c$ be as stated, let $p \in \Gamma$ and let $U_0 \subset \subset U$ be a neighborhood of $p$. We extend $c$ outside $U_0$ to obtain some positive function $c_1 \in C^{\infty}(M)$ with $c_1|_{\partial M} = 1$ and $c_1|_{\partial M} = 0$. Then the conformal metric 
\[
\tilde{g} = c_1 g
\]
satisfies 
\[
N_g = N_{\tilde{g}} = \Lambda_{\tilde{g},0,q_{\tilde{g}}}.
\]
By the assumptions, if $(\tilde{x}',\tilde{x}_n)$ are $\tilde{g}$-boundary normal coordinates near $p$, then $\tilde{g}$ also satisfies the normalization condition  
\[
\partial_{\tilde{x}_n}^j(\tilde{g}_{\alpha \beta} \partial_{\tilde{x}_n} \tilde{g}^{\alpha \beta})(x',0), \quad j \geq 1,
\]
near $p$.

Next, by computing the symbol of the pseudodifferential operator $\Lambda_{\tilde{g},0,q_{\tilde{g}}}$ near $p$, \cite[Lemma 8.7]{DKSaU} shows that knowledge of $N_g$ near $p$ determines on $\{ \tilde{x}_n = 0 \}$ the  quantities 
\[
\tilde{g}^{\alpha \beta}, \partial_{\tilde{x}_n} \tilde{g}^{\alpha \beta}, \tilde{\partial}^K \left( \frac{1}{4} \partial_{\tilde{x}_n} \tilde{k}^{\alpha \beta} + q_{\tilde{g}} \tilde{g}^{\alpha \beta} \right)
\]
for any multi-index $K$, where $\tilde{k}^{\alpha \beta} = \partial_{\tilde{x}_n} \tilde{g}^{\alpha \beta} - (\tilde{g}_{\gamma \delta} \partial_{\tilde{x}_n} \tilde{g}^{\gamma \delta}) \tilde{g}^{\alpha \beta}$. Finally an induction argument in the proof of \cite[Theorem 8.4]{DKSaU} shows that these quantities determine 
\[
\partial_{\tilde{x}_n}^j \tilde{g}^{\alpha \beta}(x',0), \quad j \geq 2.
\]
This concludes the proof.
\end{proof}

\section{Determination near the boundary}

Next we combine the boundary determination result, Lemma \ref{lemma_boundarydetermination}, with local conformal real-analyticity in order to determine the conformal class of the metric near the boundary from the knowledge of the DN map on the full boundary. The main result in this section is as follows.

\begin{Proposition}\label{near_whole_bndr}
Let $(M_1, g_1)$ and $(M_2, g_2)$ be compact locally conformally real analytic manifolds with boundary $\p M = \p M_1 = \p M_2$. Assume that $N_{g_1} = N_{g_2}$ on $\p M$. Then there is a function $c \in C^{\infty}(M_1)$ satisfying $c|_{\p M} = 1$, $\partial_{\nu} c|_{\p M} = 0$, and a diffeomorphism $F$ from a neighborhood $U$ of $\p M$ in $M_1$ onto a neighborhood of $\p M$ in $M_2$, such that 
\[
g_1 = c F^* g_2 \text{ on $U$,} \qquad F|_{\p M} = \mathrm{Id}.
\]
\end{Proposition}

The proof is in several steps. First, we have determined the formal Taylor expansion, or the jet, of $cg$ in Lemma~\ref{lemma_boundarydetermination_realanalytic} in $cg$-boundary normal coordinates, where $c$ was some smooth scaling of $g$ found in Lemma~\ref{lemma_boundarydetermination} satisfying
$$
c|_{\p M} = 1 \mbox{ and } \p_\nu c|_{\p M }=0.
$$
So far we have not used the assumption that the manifolds we are considering are locally conformally real analytic. 

We wish to determine the metric near the boundary by using the assumption of local real analyticity to show that the Taylor expansion of a conformal metric converges. Here we encounter a problem: even if the original metric $g$ is conformally real analytic in \emph{some} coordinates, it does not follow that it is conformally real analytic in $cg$-boundary normal coordinates. It follows that we cannot determine the metric near the boundary using its formal Taylor series in $cg$-boundary normal coordinates, at least in general.

To resolve this problem we use the following procedure. We first determine the Taylor expansion of $cg$ in $cg$-boundary normal coordinates. Then we construct a new set of coordinates by solving suitable Dirichlet problems for the conformal Laplacian of $cg$. The solutions of the Dirichlet problems will constitute a coordinate system which we, after a scaling, call \emph{$Z$-coordinates}.

We will see that the $Z$-coordinates induce a real analytic change of coordinates from the coordinate system where the metric is conformally real analytic to the new $Z$-coordinates. Since a real analytic transformation preserves conformal real analyticity,
$$
g_{ij}=sh_{ij}, \quad h_{ij}\in C^\omega,
$$
we conclude that after changing to $Z$-coordinates, the metric is still conformally real analytic. In particular, we conclude that if the metric is conformally real analytic in some boundary chart it is that in $Z$-coordinates. (This is analogous to the fact that changing to harmonic coordinates preserve $C^\omega$ regularity.)

We will see that if the jets of the metrics $\tilde{g}_i  = c_i g_i$ agree in $\tilde{g}_i$-boundary normal coordinates $\psi_i$, 
\[
J_x(\psi_1^{-1*}\tilde{g}_1)=J_x(\psi_2^{-1*}\tilde{g}_2),
\]
and if the DN maps agree, then we will also have
\begin{equation}
J_x((Z_1\circ\psi_1)^{-1*}\tilde{g}_1)=J_x((Z_2\circ \psi_2^{-1*})\tilde{g}_2).
\end{equation}
After these steps we have determined the Taylor expansion of a conformal metric in a coordinate system where it is conformally real analytic. 

Finally we will normalize the determinant of the conformal metric and argue that the resulting normalized metric is determined by its Taylor expansion. The whole process is illustrated in the Figure~\ref{gauges} below.

\begin{figure}[!htb]\label{gauges}
\centering
\includegraphics[scale=.5]{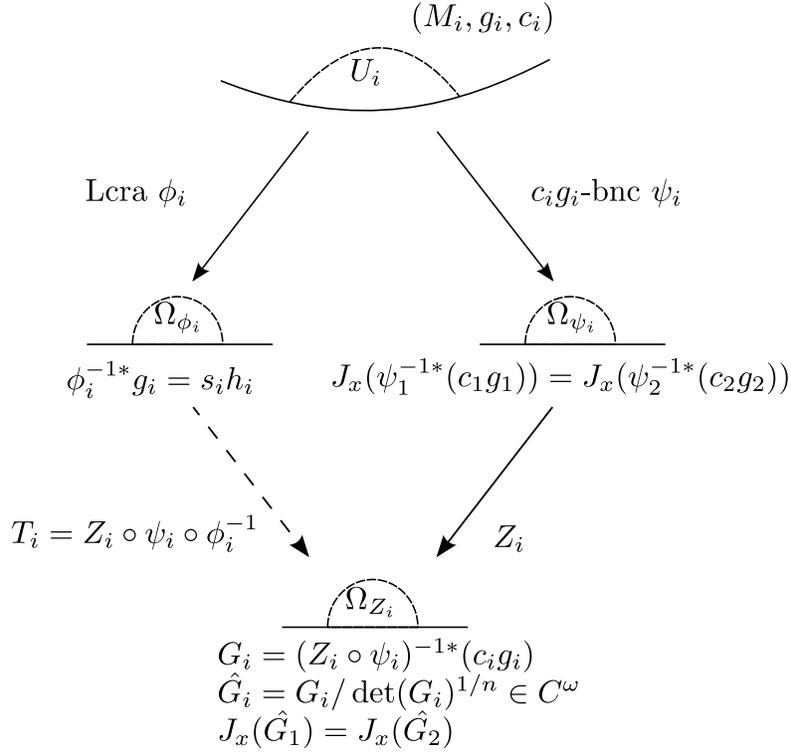}
\caption{The abbreviation "lcra" refers to a coordinate system where the metric is locally conformally real analytic and "bnc" refers to boundary normal coordinates.}
\label{fig:digraph}
\end{figure}

We remark that the procedure we described is analogous, though more involved, to the one used in the work~\cite{GuSaBar} in studying inverse problems on Einstein manifolds. There the authors use first boundary normal coordinates to determine the Taylor expansion of the metric and then use harmonic coordinates and the fact that in harmonic coordinates the metric of an Einstein manifold is real analytic.

We will now move to the details.

\subsection{Determining the jets.}
Let us make precise the conclusion of Lemma~\ref{lemma_boundarydetermination}. The lemma states that the formal Taylor expansion of $\tilde{g}=cg$ is determined on the boundary portion $\Gamma$ in $\tilde{g}$-boundary normal coordinates. This means the following. Assume that there are two Riemannian metrics $g_1$ and $g_2$ on manifolds $M_1$ and $M_2$ with common boundary portion $\Gamma\subset \p M_1, \Gamma\subset \p M_2$. Assume that the DN maps for $(M_1,g_1)$ and $(M_2,g_2)$ agree on $\Gamma \subset \p M$. Let $p\in\Gamma$.

Let then $c_1$ and $c_2$ be the smooth functions found in part (b) of Lemma~\ref{lemma_boundarydetermination} and denote $\tilde{g}_i=c_ig_i$, $i=1,2$. Let $x'=(x^{1},x^2,\ldots,x^{n-1})$ be some coordinates on $\Gamma$, and let $\psi_1$ and $\psi_2$ be the $\tilde{g}_1$- and $\tilde{g}_2$-boundary normal coordinates constructed by using the \emph{same} $x'$-coordinates. The latter implies that 
$$
\psi_1|_{\Gamma}=\psi_2|_{\Gamma}=(x',0).
$$

Let us denote by $J_p S$ the jet of a tensor field $S$ at $p$. A jet of tensor field is defined in given coordinates $(x^1,\ldots,x^n)$ as the coefficients of the Taylor expansion of the coordinate representation of the tensor field. 

With these in mind, the conclusion of Lemma~\ref{lemma_boundarydetermination_realanalytic} states that
\begin{equation}\label{Taylor_exps_agree}
J_x(\psi_1^{-1*}\tilde{g}_1) = J_x(\psi_2^{-1*}\tilde{g}_2).
\end{equation}
Here we have denoted $x=\psi_1(p)=\psi_2(p)$.

\subsection{Z-coordinates}

We next construct the new coordinate system that we call $Z$-coordinates. This is an $n$-tuple of functions constructed from (global) solutions of Dirichlet problems for the conformal Laplacian.

We will denote by $W^l$, $l=1,\ldots, n$, the corresponding $n$ functions that define a coordinate system on an open subset of $M$. Later, we will denote by $Z^l$ the coordinate representations of $W^l$. This is consistent with Figure~\ref{gauges}.

\begin{Proposition} \label{Z_coord_construction}
 Let $(M,g)$ be a Riemannian manifold with smooth boundary, with $n = \dim(M) \geq 3$. Let $p\in \p M$ and let $y' = (y^1,\ldots,y^{n-1})$ be coordinate chart on $\partial M$ near $p$. Then there exists a boundary coordinate system $W=(W^1,\ldots,W^n)$ on some open neighborhood $U\subset V$ of $p$ satisfying the following conditions:
 $$
 W^l=\frac{w^l}{w^{n+1}},\quad l=1,\ldots,n,
 $$
 where each $w^l\in C^\infty(M)$ for $l=1,\ldots,n+1$ solves 
\[
L_g w^l=0 \mbox{ in } M,
\]
and the restrictions of $w^l$ to $\Gamma = U \cap \p M$ are given by 
\begin{align*}
  w^l|_{\Gamma} &= y^l, \quad 1 \leq l \leq n-1, \\
  w^n|_{\Gamma} &= 0, \\
  w^{n+1}|_{\Gamma} &= 1.
\end{align*}
\end{Proposition}

\begin{proof}
 Let $p\in \p M$ and let $y' = (y^1,\ldots,y^{n-1})$ be a coordinate chart on $\p M$ near $p$. We first construct the functions $w^l\in C^\infty(M)$ for $l=1,\ldots,n-1$ as the solutions of the global Dirichlet problems
   \begin{align*}
  L_{g}w^l&=0 \quad \text{in $M$}, \\
   w^l &=\tilde{y}^l \quad \text{on $\partial M$},
   \end{align*}
where $\tilde{y}^l\in C^\infty(\p M)$ are some smooth continuations of the functions $y^l$, defined near $p$ on $\p M$, to functions on the whole of $\p M$.

We set $w^n$ to be the solution of
   \begin{align*}
  L_{g}w^n&=0 \quad \text{in $M$}, \\
  w^n &= \theta \quad \text{on $\partial M$},
\end{align*}
where the Dirichlet data $\theta \in C^{\infty}(\p M)$ is chosen so that 
\begin{align*}
\theta &=0 \quad \text{on $V\cap \p M$}, \\
\p_\nu w^n&\neq 0 \quad \text{on $V\cap \p M$}.
\end{align*}
Existence of such a $\theta$ is not trivial, but is guaranteed by a Runge type density argument given in Proposition~\ref{prop_runge_approximation_cauchy}.

We define the function $w^{n+1}$ to be the solution of
\begin{align*}
  L_{g}w^{n+1}&=0 \quad \text{in $M$}, \\
  w^{n+1} &= 1 \quad \text{on $\partial M$}. 
\end{align*}

We extend $y'$ to a boundary chart $y = (y', y^n)$ near $p$, and denote the coordinate representations of $w^l$ and $W^l$ by $f^l = w^l \circ y^{-1}$ and $Z^l = W^l \circ y^{-1}$. (This is consistent with Figure~\ref{gauges}.)
Let us notice that the Jacobian matrix of $Z = (Z^1, \ldots, Z^n)$ is given by 
$$
DZ|_{y(p)}=\left[\begin{array}{cc}
	I_{n-1} & 0 \\
	\p_T Z^n & \p_n Z^n\\
	\end{array}\right].
$$
Here $\p_T Z^n$ denotes the vector $(\p_1Z^n,\ldots\p_{n-1}Z^n)$ of tangential derivatives of $Z^n$. Note that the Jacobian determinant of $Z$ is non-zero at $y(p)$ if 
$$
\p_nZ^n|_{y(p)} \neq 0.
$$
But now at $y(p)$ we have
$$
\p_nZ^n=\frac{1}{f^{n+1}}\p_n f^n-\frac{f^n}{(f^{n+1})^2}\p_n f^{n+1}=\p_n f^n\neq 0,
$$
since $f^{n+1}(p)=1$ and $f^n(p)=0$. Thus there exists a neighborhood $U$ of $p$ in $M$ such that $W$ is boundary chart on $U$. This concludes the proof.
\end{proof}

The construction of the coordinate system $W$ in the proposition takes coordinates $y'$ on the boundary as input, and thus those can be chosen at will. In the next proposition, we will take the coordinates $y'$ to be $g_{\p M}$-harmonic coordinates on the boundary. Here $g_{\p M}$ denotes the induced metric on the boundary. For the existence of harmonic coordinates, see e.g.~\cite{DK, LS1}.

We will consider the functions $W^l$ introduced in the previous proposition in boundary normal coordinates $\psi$. The boundary normal coordinates $\psi$ will chosen to be the ones constructed using the same $g_{\p M}$-harmonic coordinate system on the boundary that we will use in constructing the $W$ coordinates. Thus we will have $\psi^l=w^l$, $l=1,\ldots,n-1$, on the boundary.

The coordinate representation $W^l\circ \psi^{-1}$ of $W^l$ will be denoted by $Z^l$. We call the $n$-tuple $Z = (Z^1,\ldots,Z^n)$ \emph{$Z$-coordinates}. These are functions on an open subset of the upper half plane $\mathbb{H}^n \subset \R^n$.  We also denote $f^l=w^l\circ\psi^{-1}$. 

The next proposition shows that if the Riemannian metric is locally conformally real analytic it is still that in $Z$-coordinates, at least if the coordinates on the boundary $y'$ in Proposition~\ref{Z_coord_construction} are chosen to be $g_{\p M}$-harmonic coordinates.  We formulate the proposition in the setting of the Figure~\ref{gauges}.
\begin{Proposition}\label{transformation_to_Z}
 Let $(M,g)$ be a locally conformally real analytic Riemannian manifold with boundary. Thus for given $p\in \p M$ there is a boundary coordinate chart $(U,\phi)$ on a neighborhood of $p$ where
 $$
 \phi^{-1*}g=s h, \mbox{ with } s\in C^{\infty}(\Omega_\phi) , h\in C^\omega(\Omega_\phi),
 $$
  $$
 s|_{\Omega_{\phi} \cap \{ x_n = 0 \} } \in C^{\omega}(\Omega_{\phi} \cap \{ x_n = 0 \}).
 $$
 Here $\Omega_\phi=\phi(U)\subset \mathbb{H}^n$.
 
 Let $c\in C^\infty(M)$ be a positive function on $M$ satisfying $c|_{\partial M} = 1$ and $\partial_{\nu} c|_{\partial M} = 0$. Let $\psi$ be boundary normal coordinates for $(M,cg)$ defined on a neighborhood $V\subset U$ of $p$ constructed with respect to $g_{\p M}$-harmonic coordinates on the boundary. 
 Define 
 \[
 Z^l = W^l \circ \psi^{-1}
 \]
where $W^l$ are constructed in Proposition \ref{Z_coord_construction} for $(M,cg)$ with the choice $y' = \psi'|_{\Gamma}$. Here $\psi=(\psi',\psi^n)$. Thus the coordinates $y'$ on the boundary are also $g_{\p M}$-harmonic.
 
 Then the transition function 
 $$
 T=Z\circ\psi\circ \phi^{-1}: \Omega_\phi\to \Omega_Z
 $$
 from $\phi$-coordinates to $Z$-coordinates is a real analytic diffeomorphism up to the boundary, at least after replacing $V\subset M$ by a smaller open set near $p$. We denote this set still by $V$. Above $\Omega_Z=Z(\psi(V))$. In particular, the metric $cg$ in $Z$-coordinates is conformally real analytic up to the boundary,
 $$
 Z^{-1*}\psi^{-1*}(cg)= \tilde{s} \tilde{h}, \quad \tilde{s}\in C^\infty(\Omega_Z), \tilde{h}\in C^\omega(\Omega_Z).
 $$
 Here $\tilde{s}=c|_{\psi^{-1}\circ Z^{-1}}s|_{T^{-1}}$ and $\tilde{h}=T^{-1*}h$ and $\Omega_Z$ is an open subset of $\mathbb{H}^n$ intersecting $\{x^n=0\}$.
\end{Proposition}
\begin{proof}
We show that there is a function $\gamma$ such that $\gamma w^l\circ \phi^{-1} \in C^\omega(\Omega_\phi)$, with $\gamma$ independent of $l$. Here $w^l$ are the functions used to construct the $Z$-coordinates in Proposition \ref{Z_coord_construction}. (In particular, the functions $w^l$ depend on the $\psi$-coordinates at the boundary.) 

The functions $f^l=w^l\circ\psi^{-1}$ satisfy
$$
L_{\psi^{-1*}(cg)}f^l=0 \mbox{ on } \Omega_\psi
$$
with some real analytic Dirichlet boundary conditions on an open subset of $\{x^n=0\}$. We remove the index $l$ from our notation to simplify our presentation. Applying $(\psi\circ \phi^{-1})^*$ to the above yields
\begin{align*}
0&=(\psi\circ \phi^{-1})^*L_{\psi^{-1*}(cg)}f=L_{(\phi^{-1*}\psi^*\psi^{-1*}(cg))}(w\circ \phi^{-1}) \\
&=L_{c|_{\phi^{-1}}\phi^{-1*}g}(w\circ \phi^{-1}) =L_{c|_{\phi^{-1}}s h}(w\circ \phi^{-1}) \\
&=\left(c|_{\phi^{-1}}s\right)^{-\frac{n+2}{4}}L_h\left[\left(c|_{\phi^{-1}}s\right)^{\frac{n-2}{4}}(w\circ \phi^{-1})\right].
\end{align*}
Thus
$$
L_h\left[\left(c|_{\phi^{-1}}s\right)^{\frac{n-2}{4}}(w\circ \phi^{-1})\right]=0.
$$

Note that $h$ is real analytic up to boundary by assumption and that $\left(c|_{\phi^{-1}}s\right)^{\frac{n-2}{4}}(w\circ \phi^{-1})$ has real analytic Dirichlet boundary values on $\Omega_\phi\cap\{x^n=0\}\subset \mathbb{H}^n$: First of all, we have
$$
w^n\circ \phi^{-1}(x) \text{ and } w^{n+1}\circ \phi^{-1}(x) \text{ are equal to } 0 \text{ or } 1
$$
on $\Omega_\phi\cap\{x^n=0\}$ respectively. Then, for $l=1,\ldots n-1$ we write
$$
w^l\circ \phi^{-1}=(w^l\circ \psi^{-1})\circ \underbrace{(\psi\circ \phi^{-1})}_{\in C^{\omega}(\R^{n-1})} \in C^{\omega}(\Omega_\phi\cap\{x^n=0\}).
$$
Some explanations are in order. The functions $\psi^l$ are $g_{\p M}$-harmonic on the boundary. By the coordinate invariance of $\Delta_{g_{\p M}}$, the functions $\psi^l\circ \phi^{-1}$ restricted to $\Omega_\phi\cap\{x^n=0\}$ are solutions to an elliptic equation with $C^{\omega}(\R^{n-1})$-coefficients. Thus $\psi\circ \phi^{-1}$ is real analytic on $\Omega_\phi\cap\{x^n=0\}$. We have included the details of this argument in Proposition~\ref{bndr_comeg_trans} in the appendix. We also have $w^l\circ \psi^{-1}(x)=x^l$, for $x\in \Omega_\phi\cap\{x^n=0\}$ by assumption.

Also,  $c|_{\phi^{-1}} = 1$ on the boundary and $s$ is real analytic on $\Omega_\phi\cap\{x^n=0\}$ by assumption. By these facts, the functions $\left(c|_{\phi^{-1}}s\right)^{\frac{n-2}{4}}(w^l\circ \phi^{-1})$, $l=1,\ldots,n+1$, indeed have real analytic Dirichlet boundary values on $\Omega_\phi\cap\{x^n=0\}$.

It follows from~\cite[Theorem A]{MoNi} that
$$
\left(c|_{\phi^{-1}}s\right)^{\frac{n-2}{4}}(w\circ \phi^{-1})
$$
is real analytic up to the boundary near $\phi(p)$. If necessary, we redefine $V$ as a smaller open subset near $p$ so that this function is real analytic up to the boundary on $\phi(V)$ for $l=1,\ldots,n+1$. Thus we may take $\gamma=(c|_{\phi^{-1}}s)^{\frac{n-2}{4}}$.

Now, we have that
$$
T^l=\frac{w^l}{w^{n+1}}\circ \phi^{-1}=\frac{w^l\circ\phi^{-1}}{w^{n+1}\circ\phi^{-1}}=\frac{(c|_{\phi^{-1}}s)^{\frac{n-2}{4}}  w^l\circ \phi^{-1}}{(c|_{\phi^{-1}}s)^{\frac{n-2}{4}}w^{n+1}\circ \phi^{-1}}
$$
is real analytic up to the boundary $\{x^n=0\}$. Thus $T=(T^1,\ldots,T^l)\in C^\omega(\Omega_\phi,\Omega_Z)$, where $\Omega_\phi=\phi(V)$ and $\Omega_Z=Z(\phi(V))$ are some open sets on $\mathbb{H}^n$ intersecting the set $\{x^n=0\}$.

That $cg$ is conformally real analytic (up to boundary) in $Z$-coordinates follows by noticing that
\begin{align*}
Z^{-1*}\psi^{-1*}(cg)&=c|_{\psi^{-1}\circ Z^{-1}}Z^{-1*}\psi^{-1*}\phi^*\phi^{-1*}(g) \\
&=c|_{\psi^{-1}\circ Z^{-1}}T^{-1*}(dh)=c|_{\psi^{-1}\circ Z^{-1}}s|_{T^{-1}}T^{-1*}h.
\end{align*}
Since $h\in C^\omega$ and $T$ is a $C^\omega$-diffeomorphism up to boundary, we have the claim.
\end{proof}

\subsection{Jets of $Z$-coordinates}

Next we show that if two sets of $Z$-coordinates are constructed for two Riemannian metrics $H_1$ and $H_2$ such that a) the metrics satisfy $H_1=H_2+\mathcal{O}(x_n^\infty)$ on $\{x^n=0\}$ in given boundary normal coordinates, and b) these two sets of $Z$-coordinates have the same Cauchy data on $\{x^n=0\}$, then the $Z$-coordinates agree up to infinite order in the variable $x^n$ on the set $\{x^n=0\}$.

We begin with an auxiliary lemma.

\begin{Lemma}\label{flagreeinfty}
 Let $H_1$ and $H_2$ be smooth positive definite symmetric matrix fields on an open subset $\Omega$ of the upper half plane $\mathbb{H}^n=\{(x',x_n)\in \R^n: x'\in \R^{n-1}, x^n\geq 0\}$ that intersects the boundary  $\{x_n=0\}$. Set $\Gamma=\Omega\cap \{x^n=0\}$. Assume that these matrix fields have the form
 $$
 H_i(x',x_n)=dx_n^2+h_i(x',x_n), \quad i=1,2.
 $$
 Assume that $H_1$ and $H_2$ satisfy
 $$
 H_1=H_2+\mathcal{O}(x_n^\infty) \mbox{ on } \Gamma.
 $$
 
 Let $f_1$ and $f_2$ be two functions on $\Omega$ such that
 $$
 L_{H_i}f_i=0 \quad \mbox{in } \{x^n>0\} \cap \Omega, \quad i=1,2
 $$
 with the same $C^\infty$ smooth Cauchy data on the boundary
 $$
 f_1=f_2 \mbox{ and } \p_{x_n}f_1=\p_{x_n}f_2 \mbox{ on } \Gamma.
 $$
 Then
 $$
 f_1=f_2 +\mathcal{O}(x_n^\infty).
 $$
\end{Lemma}
\begin{proof}
Since $H_1=H_2+\mathcal{O}(x_n^\infty)$ on $\Gamma$, the coefficients of $L_{H_1}$ and $L_{H_2}$ agree up to $\mathcal{O}(x_n^\infty)$ on $\Gamma$. We have
 \begin{equation}\label{Laplace_of_fi}
 L_{H_1}f_1=L_{H_1}f_2+\mathcal{O}(x_n^\infty),
 \end{equation}
 which follows by differentiating the equation
\begin{equation}\label{Eq_up_B}
 L_{H_1}f_1=0=L_{H_2}f_2.
 \end{equation}
 To see this, note first that since the Dirichlet data of $f_i$, $i=1,2$, is by assumption $C^\infty$ smooth and $H_1$ and $H_2$ are $C^\infty$ smooth up to boundary, we have that the equation above holds up to boundary (see e.g.~\cite[Theorem 5, Sec 6.3.2]{Evans}). Let $x_0\in \Gamma$ and let us calculate at $x_0$:
 \begin{align*}
 \p_{x_n}(L_{H_1}f_2)&=(\p_{x_n}L_{H_1})f_2+L_{H_1}(\p_{x_n}f_2) =(\p_{x_n}L_{H_2})f_2+L_{H_2}(\p_{x_n}f_2)\\
 &=\p_{x_n}(L_{H_2}f_2) =\p_{x_n}(L_{H_1}f_1).
 \end{align*}
 Here in the second equality we have used $H_1=H_2+\mathcal{O}(x_n^\infty)$ at $x_0$ and in the last equality we have used that~\eqref{Eq_up_B} holds at $x_0$. The higher derivatives follow similarly and by using induction.
 
 Now, the conformal Laplacian for the metrics of the given form reads
 $$
 L_{H_i}=-\p_{x_n}^2+P_{i},
 $$
 where $P_i$ is a partial differential operator containing only first order derivatives in $x_n$ and $x'$ derivatives up to second order, and whose coefficients depend only on $H_i$. Thus we can express second order $x_n$ derivatives as
 $$
 \p_{x_n}^2= P_1-L_{H_1}.
 $$
 By assumption we have
 $$
 f_1=f_2 \mbox{ and } \p_{x_n}f_1=\p_{x_n}f_2 \mbox{ on } \Gamma,
 $$
 and consequently we have on $\Gamma$ that
 \begin{align*}
 \p_{x_n}^2|_{x_n=0}f_2&=P_1|_{x_n=0}(f_2)-L_{H_1}|_{x_n=0}(f_2)\\ 
 &=P_1|_{x_n=0}(f_1)-L_{H_1}|_{x_n=0}(f_1) =\p_{x_n}^2|_{x_n=0}f_1.
 \end{align*}
 The claim follows by induction and using~\eqref{Laplace_of_fi}.
\end{proof}

The result is now a consequence of this lemma.

\begin{Proposition}\label{Zpreservjet}
Assume that the conditions of Lemma~\ref{flagreeinfty} are satisfied and let $f_i^l$, $i=1,2$ and $l=1,\ldots, n+1$, be functions satisfying the assumptions for $f_i$ in that lemma. Then the functions $Z_i^l$,
$$
Z_i^l=\frac{f_i^l}{f_i^{n+1}}, \quad i=1,2,\ \ l=1,\ldots,n,
$$
satisfy
$$
 Z_1=Z_2 +\mathcal{O}(x_n^\infty)
$$
on $\Gamma$. Here $Z_i=(Z_i^1,\ldots,Z_i^n)$, $i=1,2$.

In particular if $Z_i$ are coordinate systems, and if we have
 $$
 J_x(H_1)=J_x(H_2) \mbox{ for } x\in \Gamma,
 $$
 then
 $$
 J_x(Z_1^{-1*}H_1)=J_x(Z_2^{-1*}H_2).
 $$
 (Here $x=Z_1(x)=Z_2(x)$.)

\end{Proposition}
\begin{proof}
 The first claim follows directly from  Lemma~\ref{flagreeinfty}. The latter claim can be proven by using the chain rule to calculate the Taylor coefficients and using the knowledge that $Z_1$ and $Z_2$ have the same Taylor coefficients at the boundary.
\end{proof}

\subsection{Determination near the boundary by Taylor series}
We are ready to combine our results and newly developed tools to prove determination near the boundary. We remind the reader that the procedure is illustrated in Figure~\ref{gauges}. 

We record one more lemma whose main function is to collect all the required assumptions. Recall that $Z$-coordinates depend on given local coordinates $y'$ on the boundary as described in Proposition~\ref{Z_coord_construction}. We choose these coordinates $y'$ to be $g_{\p M}$-harmonic.
\begin{Lemma}\label{determination_if_ab}
Let $(M_i,g_i)$, $i=1,2$, be two Riemannian manifolds with a common boundary portion $\Gamma\subset \p M_1\cap \p M_2$.  Let $c_i$ be functions on $M_i$ satisfying
$$
c_i|_\Gamma=1 \mbox{ and } \p_{\nu_i}c_i|_\Gamma=0.
$$

Let $p\in \Gamma$ and assume that $\psi_i$ are $c_ig_i$-boundary normal coordinate systems on neighborhoods $U_i\subset M_i$ of $p$ that agree on $\Gamma$:
$$
\psi'_1|_\Gamma=\psi'_2|_\Gamma.
$$
Here $\psi_i=(\psi_i',\psi_i^n)$. Assume that $\psi'_1|_\Gamma$ and $\psi'_2|_\Gamma$ are $\Delta_{(g_i)_{\p M}}$-harmonic respectively on $\Gamma$. 

Let $Z_i = W_i \circ \psi_i^{-1}$ denote the $Z_i$-coordinate systems constructed in Proposition \ref{Z_coord_construction} with respect to metrics $c_i g_i$ and boundary coordinates $y' = \psi'_1|_{\Gamma} = \psi'_2|_{\Gamma}$.

Assume that
a)
The jets of the two conformal Riemannian metrics $c_ig_i$, $i=1,2$, in $Z_i\circ \psi_i$-coordinates, satisfy 
 $$
 J_{x}(Z_1^{-1*}\psi_1^{-1*}(c_1g_1))=J_{x}(Z_2^{-1*}\psi_2^{-1*}(c_2g_2)), 
 $$
 for $x$ near $x_0=Z_i\circ \psi_i(p)$ in $\{x^n=0\}$.
 
b)
There are coordinate systems $\phi_i$ on $U_i$ where the metric is conformally real analytic,
$$
\phi_i^{-1*}g_i=s_ih_i, \quad s_i\in C^\infty(\Omega_{\phi_i}), \ h_i \in C^\omega(\Omega_{\phi_i}),
$$
$$
s|_{\Omega_{\phi} \cap \{ x_n = 0 \} } \in C^{\omega}(\Omega_{\phi} \cap \{ x_n = 0 \}).
$$
Here $\Omega_\phi=\phi_i(U_i)$ is an open neighborhood of $\mathbb{H}^n$ intersecting $\{x^n=0\}$. 

If all these assumptions hold, then (after possibly shrinking $U_1$ and $U_2$) we can determine the metrics $g_1$ and $g_2$ up to a diffeomorphism $F: U_1\to U_2$ and a conformal scaling $c \in C^{\infty}(U_1)$ satisfying
 $$
 F|_\Gamma= \mathrm{Id}, \quad c|_\Gamma=1 \mbox{ and } \partial_\nu c|_\Gamma=0
 $$
in the sense that 
 $$
 g_1=c\,F^*g_2 \mbox{ on } U_1.
 $$
 (Here we have denoted $\nu_1=\nu_2=\nu$.)
\end{Lemma}
\begin{proof}
 Let us denote
 $$
 G_i=Z_i^{-1*}\psi_i^{-1*}(c_ig_i).
 $$
 We define determinant normalized metrics by
 $$
 \widehat{G}_i=\frac{G_i}{\det(G_i)^{1/n}}.
 $$
 We first note that we have
 \begin{equation}\label{G_hat_ra}
 \widehat{G}_i=\frac{\tilde{h}_i}{\det(\tilde{h}_i)^{1/n}}\in C^\omega(\Omega_Z),
 \end{equation}
 where 
 $$
 \tilde{h}_i=T_i^{-1*}h_i \in C^\omega(\Omega_Z).
 $$
 Here $\Omega_Z$ is an open neighborhood of $\mathbb{H}^n$ intersecting $\{x_n =0\}$. To see this, note that for $i=1,2$, $c_i$, $s_i$, $\phi_i$, $\psi_i$, $g_i$, $h_i$, $Z_i$ satisfy the assumptions of Proposition~\ref{transformation_to_Z}. It follows that
 $$
 T_i=Z_i\circ\psi_i\circ \phi_i^{-1}
 $$
 is a real analytic change of coordinates, $T_i \in C^\omega(\Omega_{\phi_i})$. Thus we have~\eqref{G_hat_ra} since the $s_i$ conformal factors cancel out.
 
 Since by assumption we have that
 $$
 J_{x_0}G_1=J_{x_0}G_2,
 $$
 it follows that we also have
 \begin{equation}\label{aa}
 J_{x_0}\left(\frac{G_1}{\det(G_1)^{1/n}}\right)=J_{x_0}\left(\frac{G_2}{\det(G_2)^{1/n}}\right).
 \end{equation}
 Consequently, we have
 \begin{align*}
 J_{x_0}&\left(\frac{\tilde{h}_1}{\det(\tilde{h}_1)^{1/n}}\right)=J_{x_0}\widehat{G}_1=J_{x_0}\left(\frac{G_1}{\det(G_1)^{1/n}}\right)=J_{x_0}\left(\frac{G_2}{\det(G_2)^{1/n}}\right) \\
 &=J_{x_0}\widehat{G}_2=J_{x_0}\left(\frac{\tilde{h}_2}{\det(\tilde{h}_2)^{1/n}}\right).
 \end{align*}
 Here in the third equation, we used~\eqref{aa}.
 
 Now we have that two matrix fields 
 $$
 \frac{\tilde{h}_1}{\det(\tilde{h}_1)^{1/n}} \mbox{ and } \frac{\tilde{h}_2}{\det(\tilde{h}_2)^{1/n}}
 $$
 defined on $\Omega_Z\subset \mathbb{H}^n$ that are real analytic up to boundary have the same Taylor expansion at $x_0 \in \{ x^n=0 \}$. Thus 
 $$
 \frac{\tilde{h}_1}{\det(\tilde{h}_1)^{1/n}}=\frac{\tilde{h}_2}{\det(\tilde{h}_2)^{1/n}} \mbox{ on } \Omega',
 $$
 where $\Omega'$ is an open neighborhood of $x_0$ in $\R^n$ that intersects the set $\{x^n=0\}$.
 
 Unwinding all the scalings and coordinate transformations shows first that
 $$
 G_1=\frac{\det(G_1)^{1/n}}{\det(G_2)^{1/n}}G_2
 $$
 and then
 \begin{align*}
 c_1g_1&=\psi_1^{*}Z_1^{*}G_1=\psi_1^{*}Z_1^{*}\left(\frac{\det(G_1)^{1/n}}{\det(G_2)^{1/n}}G_2\right) \\
 &=\psi_1^{*}Z_1^{*}\left(\frac{\det(G_1)^{1/n}}{\det(G_2)^{1/n}}\right)\psi_1^{*}Z_1^{*}Z_2^{-1*}\psi_2^{-1*}(c_2g_2).
 \end{align*}
 Thus defining
 $$
 F=\psi_2^{-1}\circ Z_2^{-1}\circ Z_1\circ \psi_1
 $$
 and
 $$
 c=\frac{c_2|_F}{c_1}\psi_1^{*}Z_1^{*}\left(\frac{\det(G_1)^{1/n}}{\det(G_2)^{1/n}}\right)
 $$
 ensures the main claim, i.e.\ that $g_1=cF^*g_2$, since
 $$
 F|_\Gamma=\mathrm{Id} \mbox{ and } c|_\Gamma=1,
 $$
 where the latter holds since $G_1|_{x_n=0}=G_2|_{x_n=0}$. That
 \begin{equation}\label{n_der_c}
 \p_{\nu} c|_\Gamma=0
 \end{equation}
 holds follows by noticing that the gradient of all the factors in the formula for $c$ above vanishes on $\Gamma$. More precisely, calculating in $\psi_1$-coordinates, we have that 
 $$
 \p_\nu\,\left[\psi_1^{*}Z_1^{*}\left(\frac{\det(G_1)^{1/n}}{\det(G_2)^{1/n}}\right)\right] \mbox{ on } \Gamma
 $$
 reads
 $$
 \p_{n} \left(\frac{\det(G_1)^{1/n}}{\det(G_2)^{1/n}}\right)\circ Z_1 \mbox{ on } \{x^n=0\}\cap \Omega_{\phi_1}.
 $$
 This is zero since that all the first order derivatives of 
 $$
 \frac{\det(G_1)}{\det(G_2)}
 $$ 
 vanish on $\{x^n=0\} \cap \Omega_{\phi_1}$, because the jets of $G_1$ and $G_2$ agree there. Similarly, in $\psi_1$-coordinates we have that $\p_\nu|_\Gamma c_2\circ F$ on $\Gamma$ reads
 $$
 \p_n \left[c_2 \circ (\psi_2^{-1}\circ Z_2^{-1}\circ Z_1)\right] \mbox{ on } \{x^n=0\}\cap \Omega_{\phi_1}.
 $$
 This is zero since all the first order derivatives of $c_2\circ \psi_2^{-1}$ vanish by the assumption on $c_2$ and by the fact that $\psi_2$ is a boundary normal coordinate system. These observations (and the assumptions on $c_1$) yield~\eqref{n_der_c}, which concludes the proof.
\end{proof}

The determination result for the metric near the boundary is the following.
\begin{Proposition}[]\label{det_near_bndr}
Let $(M_1, g_1)$ and $(M_2, g_2)$ be locally conformally real analytic manifolds of dimension $n = \dim(M_i) \geq 3$. Let $\Gamma$ be a nonempty open common boundary portion of both $M_i$, and assume that the DN maps coincide on $\Gamma$: for any $f \in C^{\infty}_c(\Gamma)$ we have
 $$
 N_{g_1}f|_{\Gamma}=N_{g_2}f|_{\Gamma}. 
 $$

Then for any $p \in \Gamma$ there is a diffeomorphism $F: U_1 \to U_2$ and a positive function $c \in C^{\infty}(U_1)$, for some neighborhoods $U_i$ of $p$ in $M_i$, such that
$$
g_1=c\, F^*g_2 \text{ in $U_1$}
$$
and 
\[
c|_{U_1 \cap \Gamma} =1, \ \p_\nu c|_{U_1 \cap \Gamma} =0, \ F|_{U_1 \cap \Gamma}=\mathrm{Id}.
\]
The diffeomorphism $F$ is of the form
$$
F=\psi_2^{-1}\circ Z_2^{-1}\circ Z_1\circ \psi_1,
$$
where $\psi_i$ and $Z_i$ are respectively boundary normal coordinates and $Z$-coordinates with respect to scaled metrics $c_ig_i$ introduced in Lemma~\ref{lemma_boundarydetermination} and Proposition~\ref{Z_coord_construction}.
\end{Proposition}
\begin{proof}
We only need to verify the assumptions of the previous lemma. Let $p\in \Gamma$ and let $\phi_i$ denote the coordinate systems where the metrics $g_i$ are conformally real analytic.

By Lemma~\ref{lemma_boundarydetermination}, in $c_ig_i$-boundary normal coordinates $\psi_i$ we have 
$$
J_x(\psi_1^{-1*}(c_1g_1))=J_x(\psi_2^{-1*}(c_2g_2))
$$
for $x\in \{x^n=0\}\cap \psi_i(\Gamma)$ after possibly decreasing $\Gamma$. Note that $\psi_1|_\Gamma=\psi_2|_\Gamma$. 

It follows that the induced metrics on the boundary $(g_i)_{\Gamma}$ are the same, and we can choose coordinates near $p$ on $\Gamma'\subset\Gamma$ on the boundary which are $\Delta_{(g_i)_{\Gamma}}$-harmonic for $i=1,2$. Thus, we can apply Lemma~\ref{lemma_boundarydetermination} again to have the above equality for jets in $c_ig_i$-boundary normal coordinates, where
$$
\psi_1'|_{\Gamma'}=\psi_2'|_{\Gamma'}
$$
and where $\psi'_1|_{\Gamma'}$ and $\psi'_2|_{\Gamma'}$ are $\Delta_{(g_i)_{\p M}}$-harmonic respectively on $\Gamma'$. (Recall that boundary normal coordinates are constructed by first choosing coordinates on the boundary and emitting geodesics in normal directions into the manifolds. Inverting this map gives boundary normal coordinates.)



The $Z_i$-coordinates are constructed by solving Dirichlet problems with smooth continuations of the functions $\psi'_i|_{\Gamma'}$ as the Dirichlet data. This is described in Proposition~\ref{Z_coord_construction}. Denote the continuations by $\overline{\psi}_i\in C^\infty(\p M_i)$. Since $\psi_1|_{\Gamma'}=\psi_2|_{\Gamma'}$, we can redefine $\Gamma'$ so that we still have $\bar{\psi}_1|_{\Gamma'}=\bar{\psi}_2|_{\Gamma'}$ and that the continuations satisfy  $\mbox{supp}(\overline{\psi}_i)\subset \Gamma'$. For convenience, we denote $\Gamma'=\Gamma$ in the following.

The crucial point is now to notice that since the Dirichlet-to-Neumann maps agree
$$
N_{c_1g_1} f|_\Gamma=N_{c_2g_2}f|_\Gamma,
$$
for $f \in C^{\infty}_c(\Gamma)$, we have that on $\Gamma$
$$
\nu_{c_1g_1} w_1^l = N_{c_1g_1}\overline{\psi}^l_1=N_{c_2g_2}\overline{\psi}^l_2=\nu_{c_2g_2} w_2^l.
$$
In $\psi_i$ coordinates the above reads that the coordinates $Z_i$ are constructed by using functions, call them $f_i^l$, that have the same local Cauchy data and thus satisfy the assumptions of Proposition~\ref{Zpreservjet}. 

Thus 
$$
 Z_1=Z_2 +\mathcal{O}(x_n^\infty)
$$
on the intersection of $\{x^n=0\}$ and the common domain of the $Z_i$ coordinates. It follows that after transforming into $Z_i$ coordinates it still holds that
$$
J_{x_0}(Z_1^{-1*}\psi_1^{-1*}(c_1g_1))=J_{x_0}(Z_2^{-1*}\psi_2^{-1*}(c_2g_2))
$$
(Here ${x_0}=Z_i\circ \psi_i(p_0)$.)

By assumption, the manifolds $M_1$ and $M_2$ are locally conformally real analytic. Thus the assumption b) in Lemma~\ref{determination_if_ab} holds. All the assumptions of previous proposition are now satisfied and we have the claim.
\end{proof}

We prove next that the mapping $F$ generated in the previous theorem using specifically chosen coordinates and scalings is actually independent of these.
\begin{Proposition} \label{F_indep}
  Let $F:U\to F(U)$ and $\widetilde{F}:\widetilde{U}\to F(\widetilde{U})$ be the local (conformal) diffeomorphisms generated in Proposition~\ref{det_near_bndr} with respect to 
 $$
 \psi_i \mbox{ and } \tilde{\psi}_i, \quad Z_i \mbox{ and  } \widetilde{Z}_i
 $$
 on connected open sets $U$ and $\widetilde{U}$ respectively, $i=1,2$. Assume $\p M\cap U \cap \widetilde{U}\neq \emptyset$.
 
 Assume that the Dirichlet-to-Neumann maps agree on the union of the supports of the functions $w^l_i$ and $\tilde{w}^l_i$ used to construct $Z_i$- and $\widetilde{Z}_i$-coordinates (as in Proposition~\ref{Z_coord_construction}) on $U$ and $\widetilde{U}$. Assume also that $w^l_1$ and $w^l_2$, $l=1,\ldots,n+1$, are constructed using the same Dirichlet data on $\p M$. Assume similarly for $\tilde{w}^l_1$ and $\tilde{w}^l_2$.
 
 We have
 $$
 F=W_2^{-1}\circ W_1 \mbox{ and } \widetilde{F}=\widetilde{W}_2^{-1}\circ \widetilde{W}_1
 $$
 and 
 $$
 F=\widetilde{F}
 $$ 
 on $U\cap \widetilde{U}$. Here $W_i=(W_i^1,\ldots,W_i^n)$, $i=1,2$, (similarly $\widetilde{W}_i$) are as in Proposition~\ref{Z_coord_construction}. 
 
 Also, since 
 $$
 g_1=cF^*g_2 \mbox{ and } g_1=\tilde{c}\widetilde{F}^*g_2, \mbox{ on } U\cap \widetilde{U},
 $$
 we have $c=\tilde{c}$ on $U\cap \widetilde{U}$.
 \end{Proposition}
\begin{proof}
 We show that both the functions
 $$
 \tilde{w}^l_1 \mbox{ and } c^{-\frac{n-2}{4}}F^*\tilde{w}^l_2
 $$
 satisfy the same equation 
 $$
 L_{g_1}\tilde{w}^l_1=0=L_{g_1}c^{-\frac{n-2}{4}}F^*\tilde{w}^l_2
 $$
 on $\Int(U_1\cap \widetilde{U}_1)$ with the same Cauchy data on $\Gamma\cap \widetilde{\Gamma}$. Note that there is $F$ and $c$ indeed in these formulas and not $\widetilde{F}$ and $\tilde{c}$.
 The first equality is a part of the definition of $\tilde{w}^l_1$, $l=1,\ldots, n+1$. For the second one, we simply calculate
 \begin{multline*}
 L_{g_1}c^{-\frac{n-2}{4}}F^*\tilde{w}^l_2=L_{cF^*g_2}c^{-\frac{n-2}{4}}F^*\tilde{w}^l_2 \\=c^{-\frac{n+2}{4}}L_{F^*g_2}F^*\tilde{w}^l_2 
 =c^{-\frac{n+2}{4}}F^*L_{g_2}\tilde{w}^l_2=0.
 \end{multline*}
 
 The functions $\tilde{w}^l_1$ and $\tilde{w}^l_2$ have the same Dirichlet data by assumption on $\Gamma\cap \widetilde{\Gamma}$. Since the Dirichlet-to-Neumann maps agree on their support, they also have the same Cauchy data on $\Gamma\cap \widetilde{\Gamma}$. We know that $F$ is identity on $\Gamma$ and $c=1$ on $\Gamma$. Thus $\tilde{w}^l_1$ and $c^{-\frac{n-2}{4}}F^*\tilde{w}^l_2$ have the same Dirichlet data on $\Gamma\cap \widetilde{\Gamma}$. 
 
 We also know that 
 $$
 F=\psi_2^{-1}\circ Z_2^{-1}\circ Z_1\circ \psi_1.
 $$
 Viewing the differential of $F$ using $\psi_2$ and $\psi_1$ coordinates and recalling that the jets of $Z_2$ and $Z_1$ agree on $\{x^n=0\}$ shows that
 $$
 F_*\p_{\nu_1}=\p_{\nu_2} \mbox{ on } \Gamma.
 $$
 We conclude that $\tilde{w}^l_1$ and $F^*\tilde{w}^l_2$ have the same Cauchy data on $\Gamma\cap \widetilde{\Gamma}$.
 
 Since the function $c$ also satisfies
 $$
 \p_{\nu_1} c|_\Gamma =0,
 $$
 it finally follows that $\tilde{w}^l_1$ and $c^{-\frac{n-2}{4}}F^*\tilde{w}^l_2$ have the same Cauchy data on $\Gamma\cap \widetilde{\Gamma}$. Since they solve the same elliptic equation in $U \cap \widetilde{U}$, by unique continuation~\cite[Theorem 3.3.1]{Isakov_book} (or e.g.~\cite[Section 4.3.]{Leis_book}) we have 
 $$
 \tilde{w}^l_1=c^{-\frac{n-2}{4}}F^*\tilde{w}^l_2 \mbox{ on } U\cap \widetilde{U}.
 $$
 Thus we have that
 $$
 \widetilde{W}_1^l=\frac{\tilde{w}_1^l}{\tilde{w}_1^{n+1}}=\frac{F^*\tilde{w}_2^l}{F^*\tilde{w}_2^{n+1}}=F^*\widetilde{W}^l_2 \text{ on $U \cap \widetilde{U}$}.
 $$
 (The factors $c^{-\frac{n-2}{4}}$ cancel out.)
 
 Since 
 $$
 \widetilde{F}=\tilde{\psi}_2^{-1}\circ \widetilde{Z}_2^{-1}\circ \widetilde{Z}_1\circ \tilde{\psi}_1=\widetilde{W}_2^{-1}\circ \widetilde{W}_1,
 $$
 we have that
 \begin{align*}
 \widetilde{F}&=\widetilde{W}_2^{-1}\circ\widetilde{W}_1= \widetilde{W}_2^{-1}\circ \widetilde{W}_2\circ F=F
  \end{align*}
  holding on $U\cap \widetilde{U}$. This proves the claims on $F$ and $\widetilde{F}$.
  Since
  $$
  cF^*g_2=g_1=\tilde{c}\widetilde{F}^*g_2=\tilde{c}F^*g_2 \mbox{ on } U\cap \widetilde{U},
  $$
  we also have $c=\tilde{c}$ on $U\cap \widetilde{U}$.
 \end{proof}

\begin{proof}[Proof of Proposition \ref{near_whole_bndr}]
Assume that $M_1$ and $M_2$ have a common boundary $\p M$, and that $N_{g_1} = N_{g_2}$ on $\partial M$. Proposition \ref{det_near_bndr} shows that for each $p \in \partial M$ there is a neighborhood $U_p$ in $M_1$, a local diffeomorphism $F_p$ and a positive function $c_p$ so that $g_1 = c_p F_p^* g_2$ in $U_p$. Proposition \ref{F_indep} shows that these locally defined diffeomorphisms and functions agree on the overlaps of their domains. This yields smooth maps $F: U \to M_2$ and $c: U \to \mathbb{R}_+$, where $U$ is some neighborhood of $\partial M$ in $M_1$. The mapping $F$ is injective on each $U_p$, it satisfies $g_1=cF^*g_2$ in $U$, and $F|_{\partial M} = \mathrm{Id}$.  




It remains to show that $F$ is globally injective in some neighborhood of $\partial M$ (cf.~\cite[Lemma 7.3]{KS}).
By compactness of the boundary $\inf_{p\in\p M}{\text{Inj}_i(p)}=c_i>0$, where $\text{Inj}_i(p)$ is the injectivity radius at $p$. By the continuity of $F$, and possibly after shrinking $U$, we may assume $U \subset \cup_{p \in \p M}B_1(p,r)$ and $F:U \to \cup_{p \in \p M}B_2(p,r)$ where $r< \min(c_1,c_2)$. Since we can cover $\p M$ by sets $U_p$, where $F$ is injective, we may further assume that there is $\eps$, with $r>\eps>0$, such that $F$ is injective on each ``ball of injectivity'' $B_1(p,3\eps)$ for $p\in \p M$.

Since the boundary $\p M$ is the same for both manifolds, we may assume that there is a boundary preserving diffeomorphism $\Sigma:U_1 \subset \cup_{p \in \p M}B_1(p,r)\to U_2=\cup_{p \in \p M}B_2(p,r)$. It follows that there is one-to-one correspondence with paths in $U_1$ and $U_2$. Consequently, since the metrics $g_1$ and $g_2$ are continuous, there is $C>1$, such that 
$$
\frac{1}{C}d_1^{U_1}(x,y)\leq d_2^{U_2}(\Sigma(x),\Sigma(y))\leq C d_1^{U_1}(x,y).
$$
Here 
$$
d_i^{U_i}(x,y)=\inf_{\gamma\subset U_i, \gamma:x\simeq y}l_i(\gamma).
$$
(If we knew that $M_1$ and $M_2$ are diffeomorphic by a boundary preserving global diffeomorphism, we could just use the normal distances $d_i$ of $g_i$, $i=1,2$.)

If $K$ is any compact subset of $U_1$ containing the boundary, uniform continuity implies that there is $\eps'<\eps$ such that
\[
d_2(F(p),F(q)) < \eps/C \text{ whenever } d_1(p,q) < \eps', \ p, q \in K.
\]
In particular, for some $\eps' < \eps$ we have that $B_1(p,\eps') \subset U_1$ for all $p \in \p M$ and 
\begin{equation}\label{squeeze}
F(B_1(p,\eps'))\subset B_2(p,\eps/C), \qquad p \in \p M.
\end{equation}
We redefine
$$
U_1=\cup_{p \in \p M}B_1(p,\eps').
$$
Now, if $x_1,x_2\in U_1$ are such that $F(x_1)=F(x_2)$, then $x_i\in B_1(\pi_i,\eps')$, $\pi_i\in \p M$, and we have
\begin{align*}
d_1(\pi_1,x_2)&\leq d_1(\pi_1,\pi_2)+d_1(\pi_2,x_2)< d_1^{U_1}(\pi_1,\pi_2)+\eps' \\
&\leq C d_2^{U_2}(\pi_1,\pi_2)+\eps'\leq C d_2^{U_2}(\pi_1,F(x_1))+C d_2^{U_2}(F(x_2),\pi_2)+\eps'.
\end{align*}
Now, since we have $d_2(\pi_i,F(x_i))<\eps/C<r$ by~\eqref{squeeze}, where $r$ is less than the injectivity radius of the point $\pi_i$, it follows that $d_2(\pi_i,F(x_i))=d_2^{U_2}(\pi_i,F(x_i))$. Substituting this to the above, and using~\eqref{squeeze} again, we have
$$
d_1(\pi_1,x_2)\leq  C d_2(\pi_1,F(x_1))+C d_2(F(x_2),\pi_2)+\eps'< 3\eps.
$$
Thus $x_2$ belongs to the injectivity ball of $\pi_1$, as does $x_1$. Consequently $x_1=x_2$, and $F$ is globally injective on $U_1$.
\end{proof}

\section{Green's functions agree locally}

We proceed by showing that the Green's functions for the conformal Laplacians on $(M_1,g_1)$ and $(M_2, g_2)$ agree on $U\times U$ up to a local diffeomorphism $F: U \to F(U)$ and scaling $c$ found in Proposition~\ref{det_near_bndr}. The proof is analogous to the proof of~\cite[Corollary 3.5]{GuSaBar}. 

The (Dirichlet) Green's function for the conformal Laplacian on a Riemannian manifold $(M,g)$ with boundary is the unique solution to
\begin{align*}
L_gG(y,\cdot)&=\delta_y \mbox{ on } M \\
G(y,\cdot)&=0 \mbox{ on } \p M.
\end{align*}
Here $\delta_y$ is defined with respect to Riemannian volume form so that $\int_M \delta_y f \,dV_g=f(y)$. Note that $G$ is not necessarily positive.

We first record a fact about the Schwartz kernel of the Dirichlet-to-Neumann map of the conformal Laplacian.
\begin{Lemma}\label{Schwartz_kernel}
 The Schwartz kernel $\mathcal{N}$ of $N_g$ is given for $p,p'\in \p M$, $p\neq p'$, by 
 $$
 \mathcal{N}(p,p')=\p_\nu\p_{\nu'}G(x,x')|_{x=p,\ x'=p'},
 $$
 where $\p_\nu$ and $\p_{\nu'}$ are respectively the inward pointing normal vector fields to the boundary in variable $x$ and $x'$.
\end{Lemma}
We omit the proof since it is identical to the proof of the same result for the Dirichlet-to-Neumann map of the Laplace-Beltrami operator~\cite{GuSaBar}.

\begin{Proposition}\label{Gs_agree}
 Let $(M_i,g_i)$, $i=1,2$, be two locally conformally real analytic manifolds, with a common boundary portion $\Gamma\subset \p M$, and whose Dirichlet-to-Neumann maps $N_{g_i}$ agree on $\Gamma$: for any $f \in C^{\infty}_c(\Gamma)$ we have
 $$
 N_{g_1}f|_{\Gamma}=N_{g_2}f|_{\Gamma}. 
 $$ 
 
Then for any $p \in \Gamma$ there is an open neighborhood of $U\subset M_1$ of $p$ such that the Green's functions $G_i(x,y)$ of $L_{g_i}$ satisfy
 $$
 G_1(x,y)=\frac{1}{c(x)^{\frac{n-2}{4}}c(y)^{\frac{n-2}{4}}}G_2(F(x),F(y)), \ (x,y)\in U\times U\setminus\{x = x'\}.
 $$
 Here $F$ is a local diffeomorphism and $c$ is a positive smooth function as in Proposition~\ref{det_near_bndr}.
 
 If $\Gamma=\p M$, we can take $U$ to be a neighborhood of the whole boundary.
\end{Proposition}
\begin{proof}
By Proposition~\ref{det_near_bndr} we have that for any $p\in \Gamma$ there is an open neighborhood $U\subset M_1$ of $p$ such that
$$
g_1=c F^*g_2 \mbox{ on } U,
$$
where $F:U\to F(U)$ is diffeomorphism of the form
$$
F=\psi_2^{-1}\circ Z_2^{-1}\circ Z_1\circ \psi_1
$$
and $c$ is a conformal scaling satisfying
\begin{equation}\label{boundary_rel_for_c}
c|_\Gamma =1, \ \p_{\nu_1} c|_\Gamma =0.
\end{equation}

We prove the claim by using $\psi_1$ coordinates on $U$ and $\psi_2$ coordinates on $F(U)$. Since the jets of $Z_1$ and $Z_2$ agree on $\Gamma$ by Proposition~\ref{Zpreservjet}, we have that the differential of $F$ in these coordinates on $\Gamma$ is just the identity matrix. Let us denote $\tilde{c}=c^{-\frac{n-2}{4}}$.

We first show that
\begin{equation}\label{first_identity}
\p_{\nu'}\left(\tilde{c}(x)\tilde{c}(x')G_2(F(x),F(x'))\right)=\p_{\nu'}G_1(x,x'),
\end{equation}
for $x\in U$ and $x'\in\Gamma$, $x\neq x'$. For this, let $x'\in \Gamma$. We denote
$$
T_1(x)=\p_{\nu'}G_1(x,x')
$$
and
$$
T_2(x)=\p_{\nu'}\left(\tilde{c}(x)\tilde{c}(x')G_2(F(x),F(x'))\right).
$$
We have by the diffeomorphism and conformal invariance of the conformal Laplacian that
\begin{align*}
 &L_{cF^*g_2}T_2 \\
 &=c^{-\frac{n+2}{4}} F^*L_{g_2}\left[(\tilde{c}|_{F^{-1}(\cdot)})^{-1}\p_{\nu'}\left(\tilde{c}|_{F^{-1}(\cdot)}\tilde{c}(x')G_2(\cdot,F(x'))\right)\right] \\
  &=c^{-\frac{n+2}{4}}\left(F^*\p_{\nu'}\left(\tilde{c}(x')L_{g_2}G_2(\cdot,x')\right)\right).
\end{align*}
Here the conformal Laplace operators are understood to operate on the variable $x$, which is omitted or marked as $(\cdot)$ in the equations, and not in the $x'$ variable.
For $U\ni x\neq x'$ the above equals zero. On the other hand, we have that the left hand side of the above equation equals $L_{g_1}T_2$ for $x\in U$ and thus $T_2$ solves
$$
L_{g_1}T_2=0, \mbox{ in } U\setminus\{x'\}.
$$
We also have, 
$$
L_{g_1}T_1=0, \mbox{ in } U\setminus\{x'\}.
$$

Let us then show that the Cauchy data of $T_1$ and $T_2$ agree on $\Gamma$. After this, the equation~\eqref{first_identity} follows from unique continuation~\cite{Isakov_book}. We have
\begin{align*}
 T_1|_{\Gamma\setminus\{x'\}} &=\p_{\nu'}G_1(x,x')|_{x\in \Gamma\setminus\{x'\}}=0 \\
 T_2|_{\Gamma\setminus\{x'\}} &=\p_{\nu'}\left(\tilde{c}(x)\tilde{c}(x')G_2(F(x),F(x'))\right)|_{x\in \Gamma\setminus\{x'\}}=0.
\end{align*}
The latter equation holds since $F$ is a local diffeomorphism that preserves $\Gamma$.

Since $x'\in \Gamma$, and $x\neq x'$, we have
\begin{align*}
 \p_{\nu}T_1|_{\Gamma\setminus\{x'\}}&=\p_{\nu}\p_{\nu'}G_1(x,x')|_{x\in \Gamma\setminus\{x'\}}=\mathcal{N}_1(x,x')
 \end{align*}
 by the previous lemma. Similarly, we have
 \begin{align*}
 \p_{\nu}T_2|_{\Gamma\setminus\{x'\}}&=\p_{\nu}\p_{\nu'}\left(\tilde{c}(x)\tilde{c}(x')G_2(F(x),F(x'))\right)|_{x\in \Gamma\setminus\{x'\}} \\
 &=\p_{\nu}\p_{\nu'}(G_2(F(x),F(x'))=\p_{\nu_2}\p_{\nu_2'}G_2(x,x')=\mathcal{N}_2(x,x').
\end{align*}
Some explanations are in order. In the second equality we have used that the function $\tilde{c}$ satisfies the boundary conditions~\eqref{boundary_rel_for_c}. In the second to last equality we have used the facts that the differential of $F$ is the identity matrix on the boundary in $\psi_1$ and $\psi_2$ coordinates, and that $F$ preserves $\Gamma$. We have also distinguished the boundary normal vector fields of $M_2$ by using notations $\p_{\nu_2}$ and $\p_{\nu'_2}$.

By the assumption, the Dirichlet-to-Neumann maps agree on $\Gamma$, thus their Schwartz kernels agree on $\Gamma$, and it follows that
$$
\p_{\nu}T_1|_{\Gamma\setminus\{x'\}}=\p_{\nu}T_2|_{\Gamma\setminus\{x'\}}.
$$
We have now seen that $T_1$ and $T_2$ solve the same elliptic equation and they have the same Cauchy data. It follows by unique continuation that $T_1(x)=T_2(x)$ on an open neighborhood of $\Gamma$.

We conclude the proof by using the above to show that also $G_1(x,x')$ and $\tilde{c}(x)\tilde{c}(x')G_2(F(x),F(x'))$ have the same Cauchy data for fixed $x'\in U$. This is sufficient by unique continuation: We have for $x\in U$, $x\neq x'$, that
\begin{align*}
L_{g_1}G_1(\cdot,x')&=0
\end{align*}
and
 \begin{align*}
 L_{g_1}(\tilde{c}(\cdot)\tilde{c}(x')G_2(F(\cdot),F(x')))&=\tilde{c}(x')L_{cF^*g_2}(\tilde{c}(\cdot)G_2(F(\cdot),F(x')))=0,
\end{align*}
and thus the functions in question satisfy the same elliptic equation.
Set
$$
S_1(x)=G_1(x',x) \mbox { and } S_2(x)=\tilde{c}(x')\tilde{c}(x)G_2(F(x'),F(x)).
$$
(Recall that the point $x'\in U$ is now fixed.)
We have $S_1|_{\Gamma\setminus\{x'\}}=S_2|_{\Gamma\setminus\{x'\}}=0$. By what we have proved above, we have
\begin{align*}
\p_{\nu}S_1|_{\Gamma\setminus\{x'\}}&=\p_{\nu}G_1(x',x)|_{x\in\Gamma\setminus\{x'\}}=T_1(x') \\
&=T_2(x')=\p_{\nu}\left(\tilde{c}(x')\tilde{c}(x)G_2(F(x'),F(x))\right)=\p_\nu S_2|_{\Gamma\setminus\{x'\}}
\end{align*}
Thus the Cauchy data on $\Gamma$ for $S_1$ and $S_2$ is the same. This concludes the proof.

By Proposition~\ref{F_indep} together with the proof of Proposition~\ref{near_whole_bndr} we can take $U$ to be a neighborhood of whole $\p M$.
\end{proof}

\section{Proof of the main result}
We prove our main theorem, Theorem~\ref{mainthm}. We define scaled Green's functions as follows, 
$$
H_i(x,y)=\frac{G_i(x,y)}{P_i(x)P_i(y)},
$$
for $x,y\in \Int(M_i)$, and where the functions $P_i:\Int(M_i)\to \R$, $i=1,2$, are defined as
$$
P_i(x)=\left(\int_{\p M} \left(\p_{\nu_i'}G_i(x,z')\right)^2dS(z')\right)^{1/2}.
$$
This function is nonvanishing for $x\in \Int(M_i)$. This is because if we had a point $\hat{x}\in \Int(M_i)$ such that $P_i(\hat{x})=0$, we would have that 
$$
\p_{\nu_i'}G_i(\hat{x},z')=0 \mbox{ for } z'\in \p M.
$$
Since also $G_i(\hat{x},z')=0$ for $z'\in \p M$ and $L_{g_i}G_i(\hat{x},\cdot)=0$ it would follow from elliptic unique continuation that $G_i(\hat{x},z')$ is identically zero in $M_i\setminus \{\hat{x}\}$ (recall that $M_i$ is assumed to be connected). But this would contradict the behavior of $G_i$ when $\hat{x}$ and $z'$ are close to each other. 

We have the following basic lemma, which shows that the functions $H_i(x,y)$ are real-analytic outside the diagonal in suitable coordinates. We remove the subscript $i$ from our notation for the moment.

\begin{Lemma}
Let $(M,g)$ be a locally conformally real analytic manifold with boundary. Fix points $p, p' \in \Int(M)$ with $p \neq p'$, and let $\phi$ and $\phi'$ be coordinates in some neighborhoods $U$ and $U'$ of $p$ and $p'$ so that 
\[
\phi^{-1*}g=sh \text{ and } \phi'^{-1*}g=s'h' \text{ with }h, h' \in C^{\omega}.
\]
Then the function 
\[
H(\phi^{-1}(x), \phi'^{-1}(y))
\]
is jointly real-analytic in $(\phi(U) \times \phi'(U')) \setminus \{ (x,y) \,;\, \phi^{-1}(x) = \phi'^{-1}(y) \}$.
\end{Lemma}
\begin{proof}
We need to show that both $\tilde{s}(x) \tilde{s}(y) G(x,y)$ and $\tilde{s}(x) P(x)$ are (jointly) real-analytic in suitable coordinates. Then $H(x,y)$ will also be jointly real-analytic as the quotient of real-analytic functions (recall that $\tilde{s}(x) P(x)$ is nonvanishing in $\Int(M)$).

Let first $(p_0,q_0)\in \Int(M)\times \p M$ and let $\phi$ and $\phi'$ be a coordinate chart and a boundary chart near $p_0$ and $q_0$ respectively where the metric is of the form
 $$
 \phi^{-1*}g=sh \text{ and } \phi'^{-1*}g=s'h' \text{ with }h, h' \in C^{\omega}.
 $$
 
 Let us denote
 $$
 K(p,q)=\tilde{s}(p)\tilde{s}(q)G(p,q)
 $$
 and denote by $K(x,y)$ the coordinate representation of $K(p,q)$ in coordinates $(\phi,\phi')$:
 $$
 K(x,y)=K(\phi^{-1}(x),\phi'^{-1}(y)).
 $$
 
Now the function $K(x,y)$, whose domain is a subset of $\R^n\times \mathbb{H}^n$, satisfies an elliptic equation in an open subset of $\R^{2n}$
 $$
 (L_{h}^x+L_{h'}^y)K(x,y)=0 \text{ for } x\neq y,
 $$
with real analytic coefficients and with real analytic boundary values $K(x,y)|_{\{y^n=0\}}=0$. Thus by $C^\omega$ elliptic regularity~\cite{MoNi} $K(x,y)$ is jointly real-analytic up to the boundary in an open subset of $\R^{2n}$. The same argument proves that $K(x,y)$ is jointly real-analytic also for points in $\Int(M) \times \Int(M)$, showing the required statement for $\tilde{s}(x) \tilde{s}(y) G(x,y)$.

Near $(x_0,y_0):=(\phi(p_0),\phi'(q_0))$ we can express the function $K(x,y)$ as a convergent power series
 $$
  K(x,y)=\sum_{\alpha, \beta} \frac{\p_x^{\alpha} \p_y^{\beta} K(x_0,y_0)}{\alpha! \beta!} (x-x_0)^{\alpha} (y-y_0)^{\beta}.
$$
for $(x,y)\in B(x_0,R)\times B(y_0,R)$ (i.e.\ $\sum_{\alpha, \beta} \frac{\abs{\p_x^{\alpha} \p_y^{\beta} K(x_0,y_0)}}{\alpha! \beta!} R^{\abs{\alpha}+\abs{\beta}} < \infty$) for some $R > 0$. Since $\p^{\alpha}_x K(x_0,y) = \sum_{\beta} \frac{\p_x^{\alpha} \p_y^{\beta} K(x_0,y_0)}{\beta!} (y-y_0)^{\beta}$, we can write the above in the form  
 \begin{equation*}\label{T-exp_exp}
 K(x,y)=\sum_{\alpha} \frac{\p_x^{\alpha} K(x_0,y)}{\alpha!} (x-x_0)^{\alpha}
 \end{equation*}
which holds for $(x,y) \in B(x_0,R) \times B(y_0,R)$, and we also obtain  
 \begin{equation*}\label{T-exp_exp2}
\partial_{\nu_y} K(x,y)=\sum_{\alpha} \frac{\partial_{\nu_y}\p_x^{\alpha} K(x_0,y)}{\alpha!} (x-x_0)^{\alpha}
 \end{equation*}
 for $(x,y) \in B(x_0,R) \times (B(y_0,R) \cap \{y^n=0\})$.

Using these facts, we have
 \begin{align*}
 \p_{\nu_y}(\tilde{s}(x)G(x,y)) &=\p_{\nu_y}\left(\frac{1}{\tilde{s}(y)}K(x,y)\right) \\
 &= -\tilde{s}(y)^{-2} (\partial_{\nu_y} \tilde{s}(y)) K(x,y) + \tilde{s}(y)^{-1} \partial_{\nu_y} K(x,y) \\
 &=\sum_{\alpha} \frac{1}{\alpha!}\p_x^{\alpha} \left(\p_{\nu_y}\frac{K(x_0,y)}{\tilde{s}(y)}\right) (x-x_0)^{\alpha}
 \end{align*}
 which continues to hold for $(x,y)\in B(x_0,R)\times (B(y_0,R) \cap \{y^n=0\})$. Finally, composing the above with $\phi'$ in the $y$ variable, we have
 $$
 \p_{\nu_q} (\tilde{s}(x)G(x,q)) = \sum_{\alpha} \frac{1}{\alpha!}\p_x^{\alpha} \left(\p_{\nu_q}\frac{K(x_0,q)}{\tilde{s}(q)}\right) (x-x_0)^{\alpha}
 $$
holding for $x\in B(x_0,R)$ and $q \in U_1 \cap \p M$ where $U_1$ is an open neighborhood of the boundary point $q_0$.
 
 
By using the compactness of $\p M$, we can cover the boundary with finitely many boundary charts $U_j$ so that the above holds in each $u_j$ with $R=R_j>0$. Since power expansions are unique, we have that
 $$
\p_{\nu_q} (\tilde{s}(x)G(x,q)) = \sum_{\alpha} \frac{1}{\alpha!}\p_x^{\alpha} \left(\p_{\nu_q}\frac{K(x_0,q)}{\tilde{s}(q)}\right) (x-x_0)^{\alpha}
 $$
holds for $x\in B(x_0,R)$ and $q \in \p M$ where $R=\min R_j > 0$. It also follows that the power series 
 $$
 (\p_{\nu_q} (\tilde{s}(x)G(x,q)) )^2 = \sum_{\alpha} a_{\alpha}(q) (x-x_0)^{\alpha}
 $$
has a positive radius of convergence, again denoted by $R$, independent of $q\in \p M$ (that is, $\sum_{\alpha} \norm{a_{\alpha}}_{L^{\infty}(\partial M)} R^{\abs{\alpha}} < \infty$). We can thus integrate the power series over $\p M$ and obtain that 
 \begin{align*}
(\tilde{s}(x) P(x))^2 &=\int_{\p M} (\p_{\nu_q}\tilde{s}(x)G(x,q))^2
 \end{align*}
has a convergent power series near $x_0$. Since $P$ is positive, the square root $\tilde{s}(x) P(x)$ is real analytic near $x_0$ as required.
\end{proof}

We simplify our presentation by giving the interiors of the locally conformally real analytic manifolds $(M_i,g_i)$ real analytic $C^\omega$-structures. The existence of such a structure is proven in the appendix by using $n$-harmonic coordinates. This allow us to speak about real analyticity without constantly specifying the coordinates. The transition function from the coordinates where the metric is conformally real analytic to $n$-harmonic coordinates is real analytic (see proof of Proposition~\ref{Comeg_str} in the appendix). Thus we have:
\begin{Corollary}
  Assume that Riemannian manifolds $(M_i,g_i)$, $i=1,2$, are locally conformally real analytic. Then the functions $H_i(z_i,z_i')$, $i=1,2$, are real analytic for $z_i,z_i'\in \Int(M_i)$, $z_i\neq z_i'$. 
\end{Corollary}

We will define embeddings of the interiors of $M_i$ into a Sobolev space (of negative index) to prove our main theorem, Theorem~\ref{mainthm}. Before going there, we record the following.
\begin{Proposition}\label{G_to_H}
If
\begin{equation}\label{G_rel}
 G_1(x,y)=\frac{1}{c(x)^{\frac{n-2}{4}}c(y)^{\frac{n-2}{4}}}G_2(F(x),F(y)), \ (x,y)\in U\times U\setminus\{x = y\}.
\end{equation}
 on an open set $U$ containing $\p M$, where $c$ and $F$ are the functions found in Proposition~\ref{near_whole_bndr}, then
 $$
 H_1(x,y)=H_2(F(x),F(y)),\ (x,y)\in \Int(U)\times \Int(U)\setminus\{x = y\}.
 $$
 In addition, we have that
 $$
 P_1(x)=c(x)^{-\frac{n-2}{4}}P_2(F(x)).
 $$
\end{Proposition}
\begin{proof}
Set $\tilde{c}=c^{-\frac{n-2}{4}}$. We have for $x\in \Int(U)$ that
 \begin{align}\label{bndr_trnsfm}
 &P_1^2(x)=\int_{\p M} \left(\p_{\nu_1'}G_1(x,z')\right)^2dS(z')=\tilde{c}^2(x)\int_{\p M} \left(\p_{\nu_1'}G_2(F(x),F(z'))\right)^2dS(z') \\
 \nonumber &=\tilde{c}^2(x)\int_{\p M} \left((F_*\p_{\nu_1'})G_2(F(x),z')\right)^2dS(z') \\
 \nonumber &=\tilde{c}^2(x)\int_{\p M} \left(\p_{\nu_2'}G_2(F(x),w')|_{w'=F(z')}\right)^2dS(z') \\ 
\nonumber&=\tilde{c}^2(x)\int_{\p M} \left(\p_{\nu_2'}G_2(F(x),w')\right)^2dS(w')=\tilde{c}^2(x)P^2_2(F(x)).
\end{align}
In the second equality we have used assumptions on the boundary behavior of $c$. In the fourth equality we have used that on ${\p M}$ we have
\begin{align*}
F_*\p_{\nu_1}&=\psi_{2*}^{-1}(Z_2^{-1}\circ Z_1)_* \psi_{1*}\p_{\nu_1}= \psi_{2*}^{-1}(Z_2^{-1}\circ Z_1)_*\p_{x_n} \\
&=\psi_{2*}^{-1}\p_{x_n}=\p_{\nu_2},
\end{align*}
where the second to last equality holds since the Jacobian matrix of $Z_2^{-1}\circ Z_1$ is identity matrix on ${\p M}$. We have also used that $F$ is identity on ${\p M}$ in~\eqref{bndr_trnsfm}.

Thus it follows that
$$
H_1(x,y)=\frac{G_1(x,y)}{P_1(x)P_1(y)}=\frac{\tilde{c}(x)\tilde{c}(y)G_2(F(x),F(y))}{\tilde{c}(x)\tilde{c}(y)P_2(F(x))P_2(F(y))}=H_2(F(x),F(y))
$$
for $(x,y)\in \Int(U)\times \Int(U)\setminus\{x = x'\}$.

%
%

\end{proof}

Before continuing we make some general remarks about the functions $H(x,y)$ and $P(x)$.
\begin{Lemma}\label{H_basic}
 Let $(M,g)$ be a Riemannian manifold with boundary $\p M$, and let $H(x,y)$ be the scaled (Dirichlet) Green's function
 $$
 H(x,y)=\frac{G(x,y)}{P(x)P(y)}.
 $$
 Then, when $x \in \Int(M)$ is fixed and $y \in \Int(M)$ is close to $x$, there are $c_1, c_2 > 0$ so that 
 $$
 c_1 d(x,y)^{2-n} \leq H(x,y) \leq c_2 d(x,y)^{2-n}.
 $$
 
 We have
 $$
\lim_{y \to x} \frac{H(x,y)}{d(x,y)^{2-n}} = \frac{1}{P(x)^2} c(n),
 $$
 with constant $c(n)$ depending only on the dimension $n$:
 $$
 c(n)=\frac{1}{(n-2)\omega_n}, \quad \omega_n=\frac{2\pi^{n/2}}{\Gamma(n/2)}.
 $$
\end{Lemma}
\begin{proof}
The asymptotic behavior near the diagonal of the Green's function of a uniformly elliptic operator is well known though an explicit reference seems to be hard to find. We have ($n\geq 3$)
\begin{equation}\label{Gas}
G(x,y) =\frac{1}{(n-2)\omega_n} d(x,y)^{2-n}+o(d(x,y)^{n-2}).
\end{equation}
We refer to~\cite[Lemma 25]{AS} on this result, which shows that in given local coordinates there holds
$$
G(x,y)=\frac{1}{(n-2)\omega_n} |g(x)|^{1/2} \langle g^{-1}(x)(x-y),(x-y)\rangle^{\frac{n-2}{2}}
$$
for $y$ near $x$. (Note that there is a typo in the power of $\langle g^{-1}(x)(x-y),(x-y)\rangle$ in the reference.) 
Using normal coordinates centered at $x$ then shows~\eqref{Gas}, which is a coordinate invariant equation in the leading term (the constants implied by the ''o''-notation depend on coordinates).

The results of the lemma follows since $P$ is smooth and positive in $\Int(M)$.
\end{proof}

\subsection{The embedding}
So far we have determined the metrics on an open neighborhood $U$ of the whole boundary ${\p M}$, see Proposition~\ref{near_whole_bndr}. For this we have assumed the knowledge of the Dirichlet-to-Neumann map on the whole boundary. Our next and final task is to extend the conformal mapping $F: U\to F(U)$ obtained in Proposition~\ref{near_whole_bndr} to a global conformal mapping $J:M_1\to M_2$.

Since the boundary behavior of the functions $P_i$, $i=1,2$, might pose some issues, we will work only on the interior parts of the manifold. (The functions $P_i$ are integrals of Poissons kernels, which are singular on the boundary.) Working only on the interiors is possible, since we have already determined the conformal metric near the whole boundary $\p M$.

We make the following definitions to ease our work. Below the mapping $F$ is the one found in Proposition~\ref{near_whole_bndr}. 
\begin{Definition}
We set
 \begin{align}
 N(\eps)&=\{x\in M_1: d_1(x,\p M_1)\leq \eps\} \\
 C(\eps)&=\{x\in M_1: d_1(x,\p M_1)> \eps\} \nonumber,
 \end{align}
 where $\eps >0$ is small enough so that $C(\eps)$ is connected and $N(\eps)\subset \subset U$. Then we set 
 $$
 \widehat{M}_1:=C(\eps) \mbox{ and } \widehat{U}:=U\setminus N(\eps).
 $$
 Since $F$ is a diffeomorphism $U\to F(U)$ and it maps $\p M_1$ to $\p M_2$ (i.e.\ $\p M \to \p M$), we have that $F(N(\eps))$ is a closed neighborhood of $\p M_2$ in $M_2$. We set
 $$
 \widehat{M}_2:=M_2\setminus F(N(\eps)).
 $$
 The manifolds $\widehat{M}_i$ are open manifolds.
\end{Definition}

With these definitions, we define the maps
$$
\mH_i:\widehat{M}_i\to H^s(\widehat{U})
$$
by setting
$$
\mH_1(z_1)(y)=H_1(z_1,y) \mbox{ and } \mH_2(z_2)(y)=H_2(z_2,F(y)).
$$
Here $z_i\in \widehat{M}_i$ and $y\in \widehat{U}$ and $s<2-n/2$. For $s<1-n/2$ these mapping are $C^1$, cf.~\cite{LTU}. 

A conformal mapping between locally conformally real analytic manifolds is real analytic between the interiors of the manifolds. This follows from standard regularity theory, and the proof is also given in Proposition~\ref{confComeg} in the appendix. It follows that $H_2(z_2,F(y))$ is real analytic for $z_2\neq F(y)$. 

We have the following result.
\begin{Lemma} \label{lemma_mhi_embedding}
 $\mH_i:\widehat{M}_i\to H^s(\widehat{U})$ are $C^1$-embeddings, for any $s<1-n/2$.
\end{Lemma}
\begin{proof}
The main point is to show that $\mH_i$ is injective and its derivative $D \mH_i(x): T_x \widehat{M}_i \to H^s(\widehat{U})$  is injective at any $x$. Then $\mH_i$ is an immersion (the range of $D \mH_i(x)$ is a closed split subspace since it is finite dimensional), and since $\mH_i$ actually extends as an injective $C^1$ immersion to the closure of $\widehat{M}_i$, which is compact, it is an embedding.

The proofs that $\mH_i$ and its derivative are injective are analogous to the corresponding results in~\cite{LTU}. We only prove injectivity of $\mH_2$ to show what is the idea of the proof.

Let $x_1,x_2\in \widehat{M}_2$ be such that
$$
\mH_2(x_1)=\mH_2(x_2).
$$
Thus 
$$
H_2(x_1,z)=H_2(x_2,z) \mbox{ for } z \mbox{ in the open set }F(\widehat{U}).
$$
Thus, by real analyticity, we have
$$
H_2(x_1,z)=H_2(x_2,z) \mbox{ for } z\in \widehat{M}_2\setminus \{x_1,x_2\}.
$$
Since $H_2(x,y)$ blows up only when $x=y$, we conclude that $x_1=x_2$.
\end{proof}

By Proposition~\ref{G_to_H} we have
 \begin{equation}\label{Hid}
 H_1(x,y)=H_2(F(x),F(y)),\ (x,y)\in \widehat{U}\times \widehat{U}\setminus\{x = y\}.
 \end{equation}
It follows that
\[
\mH_2 \circ F = \mH_1 \mbox{ on } \widehat{U},
\]
and thus $\mH_2|_{F(\widehat{U})}$ is a bijective map $F(\widehat{U}) \to \mH_1(\widehat{U})$ (it is injective by Lemma \ref{lemma_mhi_embedding}). This shows that 
\begin{equation}\label{F_is_J_on_hatU}
\mH_2^{-1}\circ \mH_1=F \mbox{ on } \widehat{U}.
\end{equation}

In the following result, which will imply the main theorem of this paper, the conformal diffeomorphism $F:\widehat{U}\to F(\widehat{U})$ is extended to a conformal diffeomorphism  
$$
J=\mH_2^{-1}\circ \mH_1: \widehat{M}_1\to \widehat{M}_2.
$$

\begin{Theorem}\label{extension_argument}
 Assume that we have 
 $$
 H_1(x,y)=H_2(F(x),F(y)), \mbox{ for } (x,y)\in \widehat{U}\times \widehat{U}\setminus\{x =y\}
 $$
 Then the sets $\mH_1(\widehat{M}_1)$ and $\mH_2(\widehat{M}_2)$ are identical subsets of $H^s(\widehat{U})$. Moreover, the map $J:=\mH_2^{-1}\circ \mH_1: \widehat{M}_1\to \widehat{M}_2$ is conformal.
\end{Theorem}
\begin{proof}
Define sets
$$
D_1\subset B_1\subset  \widehat{M}_1,
$$
where $B_1$ is the largest open set of points $x\in \widehat{M}_1$ such that
$$
\mH_1(x)\in \mH_2(\widehat{M}_2) \mbox{ for } x\in B_1,
$$
and where $D_1$ is the largest open set in $B_1$ where the mapping $J=\mH_2^{-1}\circ \mH_1$ is conformal with the conformal factor bounded by
\begin{align}\label{uplowbnd}
K&=\left(\frac{\max_{x_2\in \widehat{M}_2}P_2(x_2)}{\min_{x_1\in \widehat{M}_1}P_1(x_1)}\right)^{\frac{4}{n-2}} \\
\nonumber k&=\left(\frac{\min_{x_2\in \widehat{M}_2}P_2(x_2)}{\max_{x_1\in \widehat{M}_1}P_1(x_1)}\right)^{\frac{4}{n-2}}
\end{align}
from above and below respectively.
Note that $K$ and $k$ are finite and positive respectively, since the functions $P_i$ are continuous on $M_i$, $i=1,2$.

The mapping $\mH_2^{-1}\circ \mH_1$ and the notion of conformality of it are defined on $B_1$ since $\mH_i$, $i=1,2$, are $C^1$ smooth embeddings. The set $D_1$ and thus also $B_1$ contains $\widehat{U}$ and they are non-empty by~\eqref{F_is_J_on_hatU}, the assumption and Lemma~\ref{H_implies_conformal} presented after this proof. (Lemma~\ref{H_implies_conformal} gives the bounds~\eqref{uplowbnd}.)

Let $x_1$ be a boundary point of $D_1$ in $\widehat{M}_1$. Let $(p_k)\in D_1$ be a sequence such that
\begin{equation}\label{limit_seq}
\lim_{k\to \infty}p_k=x_1.
\end{equation}
Since the closure of $\widehat{M}_2$ in $M_2$ is compact we can pick subsequence of $(p_k)$ such that
$$
\lim_{k\to\infty} J(p_k)=x_2\in M_2.
$$
If $x_2\in F(U)$ then $x_1\in D_1$ and we are done. This is because $F(U)$ is open in $M_1$ and thus $F$ is a conformal diffeomorphism from a neighborhood of $x_1$ to a neighborhood of $x_2$ in this case. We remark that this is the point where the determination near the whole boundary is used. So, we assume that $x_2\notin F(U)\supset F(N(\eps))$. It follows that
$$
x_2\in \widehat{M}_2.
$$

We turn our attention to the points $x_1\in \widehat{M}_1$ and $x_2\in \widehat{M}_2$, and we will show that $x_1\in D_1$. Thus, we will have that $D_1$ is closed. Since $D_1$ is also open and $\widehat{M}_1$ connected, it will then follow that $D_1=\widehat{M}_1$.

Define $\Omega=\widehat{U}\setminus \overline{B}(x_1,\delta)$, where $\delta$ is small enough so that $\Omega$ is nonempty and connected. Define the maps $\mH_i^\Omega: \widehat{M}_i \to H^s(\Omega)$ by restricting the distributions $\mH_i(z_i)$, $z_i\in \widehat{M}_i$, to $\Omega$. 
The mappings $\mH_1^\Omega$ and $\mH_2^{\Omega}$ are real analytic on $\widehat{M}_1 \setminus \overline{\Omega}$ and $\widehat{M}_2 \setminus \overline{F(\Omega)}$, respectively. 

Since on $D_1$ we have by its definition
\begin{equation}\label{triviality}
\mH_1^\Omega=\mH_2^\Omega\circ J,
\end{equation}
and since the mappings $\mH_i^\Omega$ are continuous, it follows that there is a distribution $u\in H^s(\Omega)$ such that
$$
u=\mH_1^\Omega(x_1)=\lim_k\mH_1^\Omega(p_k)=\lim_k\mH_2^\Omega(J(p_k))=\mH_2^\Omega(x_2).
$$

Next we argue that there is an $n$-dimensional space
\begin{equation}\label{finite_param_space}
\mathcal{V}=D\mH_1^\Omega(T_{x_1}\widehat{M}_1)=D\mH_2^\Omega(T_{x_2}\widehat{M}_2).
\end{equation}
The maps $\mH_i^{\Omega}$ are also $C^1$-embeddings with the same proof as above, so the spaces $D\mH_i^\Omega(T_{x_i}\widehat{M}_i)$ are $n$-dimensional. They also coincide: let $v=D\mH_2^\Omega(V)\in D\mH_2^\Omega(T_{x_2}\widehat{M}_2)$, where $V\in T_{x_2}\widehat{M}_2$. We take a sequence $V_k\in T_{J(p_k)}\widehat{M}_2$ such that
$$
V_k\to V\in T_{x_2}\widehat{M}_2,
$$
which is bounded in the norm given by $g_2$.
Since $DJ$ is isomorphism on $D_1$, we have there a sequence is $W_k\in T_{p_k}\widehat{M}_k$ such that $DJ(W_k)=V_k$. Note that by the definition of $D_1$ this sequence is bounded in the norm given by $g_1$ by the bounds~\ref{uplowbnd}. Thus, by passing to a subsequence we may assume by compactness that
$$
W_k\to W\in T_{x_1}\widehat{M}_1.
$$
Now, by the continuity of the differentials of $\mH_i$, we have that
$$
D\mH_1^\Omega(W)=\lim_k D\mH_1^\Omega(W_k)=\lim_k D\mH_2^\Omega (V_k)=D\mH_2^\Omega(V).
$$
Inverting the role of $\widehat{M}_1$ and $\widehat{M}_2$ then proves~\eqref{finite_param_space}.

Let $\mathcal{L}\subset H^s({\Omega})$ be the orthogonal complement of $\mathcal{V}$ in $H^s(\Omega)$, and let $P:H^s(\Omega)\to \mathcal{V}$ be the orthogonal projection. We define the maps
$$
\Theta_i=P\circ \mH_i^\Omega:\widehat{M}_i\to \mathcal{V}.
$$
Now $\mH_1^{\Omega}$ is real analytic in $\widehat{M}_1 \setminus \overline{\Omega}$ and $\mH_2^{\Omega}$ is real analytic in $\widehat{M}_2 \setminus \overline{F(\Omega)}$. Also, the projection $P$ is real analytic as it is linear, the derivatives of the maps $\Theta_i$ at $x_i$, $i=1,2$, are invertible, and $\Theta_1(x_1) = \Theta_2(x_2)$. 
It follows from the inverse function theorem that there are neighborhoods $W_i$ of $x_i$ in $\widehat{M}_i$ and local real analytic inverses $K_i:\mathcal{U} \to W_i$ satisfying
$$
\Theta_i (K_i(v))=v.
$$
Here $\mathcal{U}$ is some neighborhood of the point $\Theta_1(x_1) = \Theta_2(x_2)$ in $\mathcal{V}$. (In what follows, it is useful to think $\mathcal{U}$ as a mutual domain of local parametrization of $M_1$ and $M_2$ near $x_1$ and $x_2$ respectively.) 

Thus we can represent the graphs of real-analytic functions $\mH_i^\Omega$ as graphs of real analytic mappings
\begin{equation}\label{HK_same}
\Phi_i:\mathcal{U}\to \mathcal{L}, \quad \Phi_i(v)=\mH^\Omega_i(K_i(v)).
\end{equation}
The real analytic maps $\Phi_i$ coincide in an open subset $\Theta_1(D_1)\cap \mathcal{U}$ because
\begin{align*}
\Phi_1(v)(\cdot)&=H_1(K_1(v),\cdot)=H_2(J\circ K_1(v),F(\cdot))=H_2(J\circ \Theta_1^{-1}(v),F(\cdot))\\
&=H_2(J\circ (P\circ \mH_1^\Omega)^{-1}(v),F(\cdot))
=H_2(J\circ (P\circ \mH_2^\Omega\circ J)^{-1}(v),F(\cdot))\\
&=H_2((P\circ \mH_2^\Omega)^{-1}(v),F(\cdot))=H_2(K_2(v),F(\cdot))=\mH^\Omega_2(K_2(v))(\cdot)\\
&=\Phi_2(v)(\cdot), \quad \text{for } v\in \Theta_1(D_1)\cap \mathcal{U}.
\end{align*}
Thus $\Phi_1$ and $\Phi_2$ coincide in the whole set $\mathcal{U}$ by real analyticity.


It follows that $x_1$ is an interior point of $B_1$: For $y$ belonging to the open neighborhood $K_1(\mathcal{U})\subset \widehat{M}_1$ of $x_1$ we have
$$
\mH^\Omega_1(y)(\cdot)=\mH^\Omega_1(K_1(v))(\cdot)=\Phi_1(v)(\cdot)=\Phi_2(v)(\cdot)=\mH^\Omega_2(K_2(v))(\cdot)
$$
Here $v$ is some element of $\mathcal{U}$. Thus $\mH^\Omega_1(y)\in \mH^\Omega_2(\widehat{M}_2)$. Since for $z_i\in \widehat{M}_i$ we have
\begin{equation}\label{local_equiv}
\mH_1^\Omega(z_1)=\mH_2^\Omega(z_2) \mbox{ if and only if } \mH_1(z_1)=\mH_2(z_2),
\end{equation}
it follows that $\mH_1(y)\in \mH_2(\widehat{M}_2)$ and consequently $x_1\in \Int(B_1)$. We postpone the proof of the equivalence~\eqref{local_equiv} above to Lemma~\ref{local_equiv_pf}.

Using~\eqref{local_equiv}, we have on $\mathcal{U}$ that
\begin{align*}
\Phi_2=\Phi_1 &\iff \mH_2^\Omega\circ K_2=\mH_1^\Omega\circ K_1 \iff \mH_2\circ K_2=\mH_1\circ K_1.
\end{align*}
Since $J$ is defined as $\mathcal{H}_2^{-1}\circ \mathcal{H}_1$, we have 
\[
J=K_2\circ K_1^{-1} \text{ on $K_1(\mathcal{U})$.}
\]
It also follows that
$$
\mH_1^\Omega=\mH_2^\Omega\circ J.
$$
By the formula $J=K_2\circ K_1^{-1}$ we know that $J$ is real analytic, as a composition of real analytic mappings. 
We have that
$$
H_1(x,y)=H_2(J(x),J(y))
$$
for $x$ near $x_1$ and $y\in \Omega$. (Recall that $F=J$ on $\Omega\subset \widehat{U}$.)
By real analyticity of the both sides in the $y$ variable we have that this holds for $x,y$ near $x_1$. 

By Lemma~\ref{H_implies_conformal} below, we have that $J$ satisfies the equation of conformal mapping
$$
g_1(x)=\hat{c}(x)J^*g_2(x)
$$
near $x_1$ with conformal factor
$$
\hat{c}(x)=\left(\frac{P_2(J(x))}{P_1(x)}\right)^{\frac{4}{n-2}}.
$$
Since the conformal factor $\hat{c}$ satisfies the bounds~\eqref{uplowbnd} we have that $x_1$ in $D_1$. This concludes the proof. We have found a mapping $J$ that extends the conformal mapping $F:U\to F(U)$ into a global conformal mapping.
\end{proof}

Finally, we combine our results to prove our main theorem, Theorem~\ref{mainthm}.
\begin{proof}[Proof of Theorem~\ref{mainthm}]
By Propositions~\ref{det_near_bndr} and \ref{near_whole_bndr} we can find a local conformal mapping $F:U\to F(U)$ where $U$ is a neighborhood of $\p M$ in $M_1$ and $F(U)$ is a neighborhood of $\p M$ in $M_2$. By Propositions~\ref{Gs_agree} and~\ref{G_to_H} we have
 $$
  H_1(x,y)=H_2(F(x),F(y)), \quad (x,y)\in \widehat{U}\times\widehat{U}\setminus \{x=y\}.
 $$
 Theorem~\ref{extension_argument} above now concludes the proof.
\end{proof}

We are left to prove the lemmas used in the proof of Theorem~\ref{extension_argument}.
\begin{Lemma}\label{H_implies_conformal}
 Let $U_i \subset M_i$ be open, and let $J:(U_1,g_1)\to (U_2,g_2)$ be a diffeomorphism satisfying
 $$
 H_1(x_1,y)=H_2(J(x_1),J(y))
 $$
in $(U_1 \times U_1) \setminus \{ (x, x) \,;\, x \in U_1\}$. Then $J$ is conformal on $U_1$, $g_1=cJ^*g_2$, and the conformal factor is given by
 $$
 c(x_1)=\left(\frac{P_2(J(x_1))}{P_1(x_1)}\right)^{\frac{4}{n-2}}, \quad x_1\in U_1.
 $$
\end{Lemma}
\begin{proof}
By the behavior (see Lemma~\ref{H_basic}) of $H_1(x_1,y)$ and $H_2(J(x_1),J(y))$ when $y$ is near $x_1$, we see that
$$
P_1(x_1)^2d_1(x_1,y)^{n-2}=P_2(J(x_1))^2d_2(J(x_1),J(y))^{n-2} + o(d_1(x_1,y)^{n-2}).
$$
Let $V\in T_{x_1} U_1$ and let $y(t)$ be a $g_1$-geodesic with $y(0)=x_1$ and $\dot{y}(0)=V$. We have
$$
P_1(x_1)^2|V|_{g_1(x_1)}^{n-2}=P_2(J(x_1))^2|J_*V|_{g_2(J(x_1))}^{n-2} + o(t^{n-3}).
$$
Thus, by the elementary polarization identity, we have that 
$$
g_1(x_1)=c(x_1)J^*g_2(x_1),
$$
where
\[
c(x_1)=\left(\frac{P_2(J(x_1))}{P_1(x_1)}\right)^{\frac{4}{n-2}}. \qedhere
\]
\end{proof}

\begin{Lemma}\label{local_equiv_pf}
 In the setting of the proof of Theorem~\ref{mainthm} above
 \begin{equation}\label{local_equiv2}
\mH_1^\Omega(z_1)=\mH_2^\Omega(z_2) \mbox{ if and only if } \mH_1(z_1)=\mH_2(z_2).
\end{equation}
\end{Lemma}
\begin{proof}
Note that if
$$
H_1(z_1,y)=H_2(z_2,F(y)), \quad y\in \Omega,
$$
then by real analyticity of $H_1(z_1,\cdot)$ and $H_2(z_2,F(\cdot))$ we have the above for $y\in \widehat{U}\setminus\{z_1, F^{-1}(z_2)\}$. We also must have that $F(z_1)=z_2$, due to the diagonal behavior of the functions $H_i$ (see Lemma~\ref{H_basic}). Thus the function
$$
u=H_1(z_1,\cdot)-H_2(z_2,F(\cdot))
$$
is identically zero outside its singular support $\{z_1\}$. We argue that $u$ is actually the zero distribution. We calculate using $g_1=cF^*g_2$ holding on $\widehat{U}$, together with identity $P_1(x)=c(x)^{-\frac{n-2}{4}}P_2(F(x))$ of Proposition~\ref{G_to_H}, as follows
\begin{align*}
L_{g_1}P_1(z_1)&P_1(\cdot)u(\cdot)=L_{g_1}G_1(z_1,\cdot)\\ 
&-c(z_1)^{-\frac{n-2}{4}}L_{g_1}\left[P_2(F(z_1))c(\cdot)^{-\frac{n-2}{4}}P_2(F(\cdot))H_2(F(z_1),F(\cdot))\right] \\
&=\delta_{g_1,z_1}-c(z_1)^{-\frac{n-2}{4}}L_{cF^*{g_2}}\left[c(\cdot)^{-\frac{n-2}{4}}G_2(F(z_1),F(\cdot))\right] \\
&=\delta_{g_1,z_1}-c(z_1)^{-\frac{n-2}{4}}c(\cdot)^{-\frac{n+2}{4}}L_{F^* g_2}G_2(F(z_1),F(\cdot)) \\
&=\delta_{g_1,z_1}-c(z_1)^{-\frac{n-2}{4}} c(\cdot)^{-\frac{n+2}{4}} F^*(\delta_{g_2,F(z_1)}) =\delta_{g_1,z_1}-\delta_{g_1,z_1}=0.
\end{align*}
In the second to last equality we have used again that $g_1=cF^*g_2$. The above holds on $\widehat{U}$. We conclude by unique continuation~\cite{Isakov_book} that $P_1(z_1)P_1(x)u(x)$, and thus $u(x)$, are indeed the zero distributions. The equivalence~\eqref{local_equiv2} follows.
\end{proof}

\appendix

\section{Real analytic structure} \label{section_realanalytic_structure}

We show that a locally conformally real analytic manifold (without boundary) admits a real analytic structure. The proof is by using $n$-harmonic coordinates~\cite{LS1}. We also show that if a conformally real analytic manifold has a boundary, then the boundary is real analytic as a Riemannian manifold. In addition, we record the fact that a conformal mapping between locally conformally real analytic manifolds is real analytic. 

An $n$-harmonic coordinate system is a conformal generalization of harmonic coordinate system and also a generalization of isothermal coordinates. We refer to~\cite{LS1,LS2,JLS} on details on $n$-harmonic coordinates.

\begin{Proposition}\label{Comeg_str}
 Let $(M,g)$ be a locally conformally real analytic manifold \emph{without} boundary whose transition functions are $C^\infty$ smooth (i.e. $M$ as a manifold is $C^\infty$ smooth). Then $M$ admits a real analytic $C^\omega$-structure compatible with the original $C^\infty$-structure. 
 
 The $C^\omega$-structure can be given by transforming from coordinates where the metric is locally conformally real analytic to $n$-harmonic coordinates. The metric is conformally real analytic in the coordinates of the new atlas.
\end{Proposition}
\begin{proof}
We only need to show that we can replace the $C^\infty$ smooth transition functions of $M$ by real analytic ones by transforming to $n$-harmonic coordinates. For this let $\phi$ and $\tilde{\phi}$ be two coordinate charts whose domains contain a point $x_0 \in M$ and the metric $g$ of $M$ is conformally real analytic in both of these coordinates.
 
 Let $T$ and $\widetilde{T}$ be transition functions from these coordinates to $n$-harmonic coordinates near $x_0$. We show that these mappings are $C^\omega$ smooth since their components satisfy a quasilinear elliptic equation on $\R^n$: 
 $$
 \Delta_nT^k=-\abs{g}^{-1/2} \partial_j (\abs{g}^{1/2} g^{jm} (g^{ab} \partial_a T^k \partial_b T^k)^{\frac{n-2}{2}} \partial_m T^k)=0, \ k=1,\ldots,n.
 $$
 See~\cite{LS1} for details on the $n$-harmonic equation. (Similarly for $\widetilde{T}$.) In~\cite{LS1} it is also proven that each $T^k$ is $C^\infty$ smooth.
 
 We establish the assumptions of~\cite[Corollary 1.4]{KN} for a $C^\omega$ regularity result for these type of equations with $C^\omega$ coefficients in our case where each $dT^k\neq 0$. The latter holds since otherwise the Jacobian determinant of $T$ would be zero contradicting the invertibility of $T$. 
 
 For the reference~\cite[Corollary 1.4]{KN}, we set for simplicity of the notation $T^k=u$, and  consider the above as a nonlinear elliptic equation: 
 $$
 \sum_j \p_j F^j(x,d u)=0,
 $$
 on $U=\Omega\times (\R^n\setminus B_\eps(0))\subset \R^n\times \R^{n}$, with
 $$
 F^j(x,p)=\abs{g}^{1/2} g^{jm} (g^{ab} p_a p_b)^{\frac{n-2}{2}} p_m.
 $$
 Here $\Omega$ is the domain of each $T^k$ and we have restricted each $F^j$ to $\R^n\setminus B_\eps(0)$, $\eps>0$, by the facts that each $du=dT^k\neq 0$ and that $dT^k$ continuous on $\Omega$. (We could also consider that the mapping $T$ satisfy an nonlinear elliptic system in the spirit of~\cite{KN}, but since the equations for $T^k$ are uncoupled, it makes no difference.)

 In the notation of~\cite[Corollary 1.4]{KN}, we set
 \begin{align*}
 a^{ij}=\p_{p_j}F^i(x,p)&=|g|^{1/2}g^{im} \Big(\frac{n-2}{2} (g^{ab}p_ap_b)^{\frac{n-4}{2}}g^{cd}(\delta^{j}_cp_d+\delta^{j}_dp_c)p_m \\
 &\quad \quad \quad \quad +(g^{ab}p_ap_b)^{\frac{n-2}{2}}\delta_m^j\Big) \\
 &=|g|^{1/2}(g^{ab}p_ap_b)^{\frac{n-4}{2}}\Big((n-2)g^{jd}p_dg^{im}p_m+ (g^{ab}p_ap_b)g^{ij}\Big).
 \end{align*}
 Here, as usual, repeated indices are summed over. We define a (scalar) symbol $A(\xi)$, for $0\neq \xi\in \R^n$, as in~\cite{KN}:
 $$
 A(\xi)=a^{ij}\xi_i\xi_j=|g|^{1/2}(g^{ab}p_ap_b)^{\frac{n-4}{2}}\big((n-2)(p_ag^{ai}\xi_i)^2+(g^{ab}p_ap_b)g^{ij}\xi_i\xi_j\big).
 $$
 We trivially have
 $$
 A(\xi)\geq |g|^{1/2}(g^{ab}p_ap_b)^{\frac{n-4}{2}}|p|^2|\xi|^2
 $$
 and thus the assumptions~\cite[Definition 1.1]{KN} are satisfied on $U$, where $|p|\geq \eps$. Thus $u=T^k$, and consequently $T$, are real analytic.
 
 
 We have that the components
 $$
 \widetilde{T}^k\circ \tilde{\phi}\circ \phi^{-1}\circ T, \quad k=1,\ldots,n,
 $$
 of the mapping $(\widetilde{T}\circ \tilde{\phi})\circ (T\circ \phi)^{-1}$ also satisfy an $n$-harmonic equation with $C^\omega$ coefficients in a suitable subset of $\R^n$: first we see that
 \begin{align*}
 \Delta^n_{T^{-1*}\phi^{-1*}g}&\widetilde{T}^k\circ \tilde{\phi}\circ \phi^{-1}\circ T^{-1}=(\tilde{\phi}\circ \phi^{-1}\circ T^{-1})^*\Delta^n_{\tilde{\phi}^{-1*}g}\widetilde{T}^k=0,
 \end{align*}
 where $\Delta^n$ is the $n$-Laplacian~\cite{LS1}. Then we note that
 $$
 T^{-1*}\phi^{-1*}g=s|_{T^{-1}} T^{-1*}h,
 $$
 since $\phi$ was such that $\phi^{-1*}g=s\, h$ with $s\in C^\infty$ and $h\in C^\omega$,
 and use this with the fact that the $n$-harmonic equation is invariant under conformal scalings~\cite{LS1}. Thus
 $$
 \Delta^n_{T^{-1*}h}\widetilde{T}^k\circ \tilde{\phi}\circ \phi^{-1}\circ T^{-1}=0,
 $$
 where $T^{-1*}h\in C^{\omega}$. As already mentioned above this can be seen as a quasilinear elliptic equation~\cite{LS1} for $\widetilde{T}^k\circ \tilde{\phi}\circ \phi^{-1}\circ T^{-1}$ with $C^\omega$ coefficients. Thus it follows that the transition function $(\widetilde{T}\circ \tilde{\phi})\circ (T\circ \phi)^{-1}$ is $C^\omega$ regular~\cite[Corollary 1.4]{KN}.
 
 Defining new coordinate charts as $\phi'=T\circ \phi$ and $\tilde{\phi}'=\widetilde{T}\circ \tilde{\phi}$, and extending this procedure over all charts of $M$ where the metric is conformally real analytic gives an $C^\omega$-structure for $M$, which is $C^\infty$ compatible with the original atlas. (The latter follows since $T$ and $\widetilde{T}$ are actually even $C^\omega$ local diffeomorphisms of $\R^n$ as noted above.)
 
 In the new coordinates the metric is still conformally real analytic:
 $$
 \phi'^{-1*}g=s' h', \quad s'\in C^\infty \mbox{ and } h'\in C^\omega,
 $$
 and similarly for $\tilde{\phi}'$.
\end{proof}

\begin{Proposition}\label{bndr_comeg_trans}
Let $(M,g)$ be a locally conformally real analytic manifold \emph{with} boundary and let $g_{\p M}$ be the induced metric on the boundary $\p M$.

 Let $x'=(x^1,\ldots,x^{n-1})$ be $g_{\p M}$-harmonic coordinates on a neighborhood $\Gamma$ of $p\in \p M$ \emph{on} the boundary, and let $\phi=(\phi^1,\ldots,\phi^{n-1},\phi^n)$ be coordinates on $U\subset M$ near the boundary point $p$ such that
 $$
 \phi^{-1*}g=s h, \mbox{ with } s\in C^{\infty}(\Omega) , h\in C^\omega(\Omega),
 $$
  $$
 s|_{\Omega \cap \{ x_n = 0 \} } \in C^{\omega}(\Omega \cap \{ x_n = 0 \}).
 $$
 Here $\Omega=\phi(U)\subset \mathbb{H}^n$. Denote $\Sigma=\Omega \cap \{ x_n = 0 \}$ and denote $\phi'=(\phi^1,\ldots,\phi^{n-1}):U\cap \p M \to \Sigma$.
 
 Then, the induced transition function $S=x'\circ\phi'^{-1}: \Sigma\to x'(\Gamma)$ is real analytic. In particular, changing to $g|_{\p M}$-harmonic coordinates on the boundary gives $\p M$ a real analytic atlas. The coordinate representations of $g_{\p M}$ in the coordinates of this atlas are real analytic.  
\end{Proposition}
\begin{proof}
 Let $p\in \p M$ and let $x'$ and $\phi$ be as in the assumption with
 $$
 \phi^{-1*}g=sh.
 $$
Now, $\phi'=(\phi^1,\ldots,\phi^{n-1})$ is a coordinate chart on $U\cap \p M$ (by the definition of boundary chart). 

Let us decompose the matrix field $h$ as
$$
h=\left[\begin{matrix}
h^{(n-1)} & h_{\#n} \\
h_{n\#} & h_{nn}
\end{matrix}\right],
$$
where $h^{(n-1)}$ is the $(n-1)\times (n-1)$ upper left block of whole $h$-matrix field. The coordinate representation of the boundary metric $g_{\p M}$ in $\phi'$-coordinates reads 
$$
\phi'^{-1*}g_{\p M}=s|_\Sigma h^{(n-1)}|_\Sigma.
$$
 
 
 Let $S=x'\circ\phi'^{-1}$ be the transition function $\Sigma\to x'(\Gamma)$. We have
  $$
 \phi'^*\Delta_{s|_{\Sigma}h^{n-1}|_{\Sigma}}S^l=\Delta_{\phi'^*\left(s|_{\Sigma}h^{n-1}|_{\Sigma}\right)}x^l=\Delta_{g_{\p M}} x^l=0, \quad l=1,\ldots n.
 $$
 (Here ``$\Delta$'' is the $(n-1)$-dimensional Laplace-Beltrami operator.)
 Thus, each component function $S^l$ of $S$ is real analytic as a solution of an elliptic equation with real analytic coefficients, see e.g.~\cite[Appendix J]{Besse}. That transforming to harmonic coordinates gives $\p M$ a real analytic atlas follows as in the proof of previous proposition.
 \end{proof}

\begin{Proposition}\label{confComeg}
 A conformal mapping $F$ between locally conformally real analytic manifolds $(M,g)$ and $(N,h)$ is a real analytic diffeomorphism between the interiors of the manifolds.
\end{Proposition}
\begin{proof}
The proof is similar to the one used to prove smoothness of conformal mappings in~\cite{LS1}. Let $p\in M$ and let $\phi'$ and $\tilde{\phi}'$ be coordinates near $p$ and $F(p)$ such that the metrics $g$ and $h$ are locally conformally real analytic in the coordinates respectively.

The components
$$
(\tilde{\phi'})^k\circ F \circ \tilde{\phi}^{-1}, \quad k=1,\ldots,n,
$$
of the mapping $F$ in the introduced coordinates satisfy
$$
\Delta^n_{\tilde{\phi}^{-1*}g}(\tilde{\phi'})^k\circ F \circ \tilde{\phi}^{-1}=0.
$$
See~\cite{LS1} for details on this argument. Since the matrix field $\tilde{\phi}^{-1*}g$ is conformally real analytic, and $\Delta^n$ is conformally invariant, this proves the claim by elliptic regularity of quasilinear equations, see~\cite[Corollary 1.4]{KN}.
\end{proof}

\section{Density} \label{sec_density}

In this section we assume that $(M,g)$ is a compact connected oriented Riemannian manifold with boundary, and $\dim(M) \geq 2$. We also assume that $q \in C^{\infty}(M)$ and that $0$ is not a Dirichlet eigenvalue of $-\Delta+q$ in $M$ (i.e.\ the only solution $u \in H^1_0(M)$ of $(-\Delta+q)u = 0$ is $u=0$).

The next Runge approximation type result implies that one may solve $(-\Delta+q)u = 0$ in $M$ while approximately prescribing the Cauchy data of $u$ on a strict open subset of $\partial M$.

\begin{Proposition} \label{prop_runge_approximation_cauchy}
Let $\Gamma$ be a nonempty open subset of $\partial M$ satisfying $\overline{\Gamma} \neq \partial M$. Then the set 
\[
\{ \partial_{\nu} u|_{\overline{\Gamma}} \,;\, u \in C^{\infty}(M), \ (-\Delta+q)u = 0 \text{ in $M$}, \ u|_{\Gamma} = 0 \}
\]
is dense in $C^{\infty}(\overline{\Gamma})$.
\end{Proposition}

As a consequence, for any $\varepsilon > 0$ there exists $w^n \in C^{\infty}(M)$ satisfying 
\[
(-\Delta+q)w^n = 0, \quad w^n|_{\Gamma} = 0,
\]
and $\norm{\partial_{\nu} w^n|_{\Gamma} - 1}_{L^{\infty}(\Gamma)} < \varepsilon$.

The proof of Proposition \ref{prop_runge_approximation_cauchy} requires a duality argument, Dirichlet problems with boundary data in negative order Sobolev spaces, and unique continuation.

\begin{Lemma} \label{lemma_dirichlet_problem_negative_sobolev}
Let $P$ the operator $C^{\infty}(\partial M) \to C^{\infty}(M)$ which maps $f$ to the unique solution of $(-\Delta+q)u = 0$ in $M$ satisfying $u|_{\partial M} = f$.

\begin{tehtratk}
\item[(a)] 
For any $s \in \mR$, $P$ has a bounded extension 
\[
P^{(s)}: H^s(\partial M) \to D^{s+1/2}(M)
\]
where $D^{s+1/2}(M)$ is a Hilbert space which has $C^{\infty}(M)$ as a dense subspace and is continuously contained in $H^{s+1/2}(M)$, and there are bounded trace operators 
\[
\gamma^{(s)}_j: D^{s+1/2}(M) \to H^{s-j}(\partial M), \qquad j=0,1,
\]
extending $\gamma_0: u \mapsto u|_{\partial M}$ and $\gamma_1: u \mapsto \partial_{\nu} u|_{\partial M}$ acting on $C^{\infty}(M)$. For any $f \in H^s(\partial M)$, the function $u = P^{(s)} f \in H^{s+1/2}(M)$ solves $(-\Delta+q)u = 0$ in $M$ in the sense of distributions, has boundary value $\gamma^{(s)}_0 u = f$, and satisfies $\gamma_1^{(s)} u \in H^{s-1}(\partial M)$.

\item[(b)]
If $f \in H^s(\partial M)$ is $C^{\infty}$ near some boundary point $p$, then there is a neighborhood $U$ of $p$ in $M$ so that $u = P^{(s)} f$ is $C^{\infty}$ in $U$.

\item[(c)]
One has the following unique continuation statement: if $\Gamma \subset \partial M$ is a nonempty open set and if $v = P^{(s)} f$ for some $f \in H^s(\partial M)$ satisfies 
\[
v|_{\Gamma} = \partial_{\nu} v|_{\Gamma} = 0,
\]
then $v = 0$.
\end{tehtratk}
\end{Lemma}
\begin{proof}
(a) follows from \cite[Chapter 2]{LionsMagenes}, with the final result given in Section 7.3 of Chapter 2.

For (b), writing $f = \chi f + (1-\chi)f$ where $\chi \in C^{\infty}(\partial M)$ satisfies $\chi = 1$ near $p$ and $\chi f \in C^{\infty}(\partial M)$, and noting that $P^{(s)}(\chi f) \in C^{\infty}(M)$, it is enough to show that $u = P^{(s)}((1-\chi) f)$ is $C^{\infty}$ near $p$. We use a parametrix for this boundary value problem \cite[Theorem 12.6]{Taylor2}: there is a collar neighborhood $\mathcal{C} = [0,1] \times \partial M$ and a distribution $u^{\sharp}$ in $\mathcal{C}$ with 
\[
u^{\sharp} - u \in C^{\infty}(\mathcal{C})
\]
so that, using coordinates $(y,x)$ in $\mathcal{C}$, one has $u^{\sharp} = Q((1-\chi)f)$ where $Qh(y,\,\cdot\,) = Q(y) h$ and $Q(y)$ is a pseudodifferential operator on $\partial M$ whose symbol $q(y,x,\xi)$ satisfies, for some $C_1 > 0$, 
\begin{gather*}
q(y,x,\xi) = b(y,x,\xi) e^{-C_1 y \langle \xi \rangle}, \\
y^k D_y^l b(y,\,\cdot\,) \text{ is bounded in $S^{-k+l}_{1,0}(\partial M)$ for $y \in [0,1]$.}
\end{gather*}
(\cite{Taylor2} does not discuss $Q$ acting on negative order Sobolev spaces, but this can be justified since using \cite[Lemma 12.5]{Taylor2} the operator $D_y^l A(y)$ with $A(y)$ as in \cite[Proposition 12.4]{Taylor2} maps $H^{-r}(\partial M)$ to the mixed norm space $L^2_y H^{-r+1/2-l}_x(\mathcal{C})$ for $r \geq 0$.) Suppose $\psi \in C^{\infty}(\partial M)$ satisfies $\psi = 1$ near $p$ and $\mathrm{supp}(\psi) \subset \subset \{ \chi = 1 \}$. It is enough to show that 
\[
\psi(x) Q( (1-\chi(x)) f) \in C^{\infty}(\mathcal{C}).
\]
Since each $Q(y)$ is in $\mathrm{Op}(S^0_{1,0}(\partial M))$ and thus has the pseudolocal property, the operators $\psi(x) \circ Q(y) \circ  (1-\chi(x))$ are in $\mathrm{Op}(S^{-\infty}_{1,0}(\partial M))$. Moreover, these operators and their $y$-derivatives up to a fixed order have uniform symbol bounds for $y \in [0,1]$. This shows the required result.

To prove (c), we note that by (b) the function $v$ is $C^{\infty}$ near any compact subset $K$ of $\Gamma$. Thus $v$ is a smooth solution of $(-\Delta+q) v = 0$ near $K$ with $v|_{\Gamma} = \partial_{\nu} v|_{\Gamma} = 0$, and the unique continuation principle~\cite{Isakov_book} implies that $v \equiv 0$ near $K$. Since $v$ solves $(-\Delta+q)v = 0$ in $M$, it follows that $v \equiv 0$.
\end{proof}

\begin{Lemma} \label{lemma_runge_approximation_sobolev}
Let $\Gamma$ be a nonempty open subset of $\partial M$ with $\overline{\Gamma} \neq \partial M$. Let also $s > 0$. Then the set 
\[
S = \{ \partial_{\nu} u|_{\Gamma} \,;\, u \in H^{s+3/2}(M), \ (-\Delta+q)u = 0 \text{ in $M$}, \ u|_{\Gamma} = 0 \}
\]
is dense in $H^s(\Gamma)$.
\end{Lemma}
\begin{proof}
Let $T$ be a bounded linear functional on $H^s(\Gamma)$ satisfying 
\[
T(\partial_{\nu} u|_{\Gamma}) = 0, \qquad u \in S.
\]
It is enough to show that $T = 0$. Define 
\[
\tilde{T}: H^s(\partial M) \to \mR, \ \ \tilde{T}(h) = T(h|_{\Gamma}).
\]
Then $\abs{\tilde{T}(h)} \leq C \norm{h|_{\Gamma}}_{H^s(\Gamma)} \leq C \norm{h}_{H^s(\partial M)}$, so $\tilde{T}$ is bounded on $H^s(\partial M)$ and by duality there exists $f \in H^{-s}(\partial M)$ satisfying 
\[
\tilde{T}(h) = (f, h), \qquad h \in H^s(\partial M).
\]
(This is the place where the negative order Sobolev spaces appear in this argument.) If $\mathrm{supp}(h) \subset \partial M \setminus \overline{\Gamma}$, then $(f,h) = \tilde{T}(h) = T(h|_{\Gamma}) = 0$. It follows that $\mathrm{supp}(f) \subset \overline{\Gamma}$.

Choose a sequence $(f_j) \subset C^{\infty}(\partial M)$ with $f_j \to f$ in $H^{-s}(\partial M)$. Let $v_j = P f_j$ and $v = P^{(-s)} f$ with $P^{(-s)}$ as in Lemma \ref{lemma_dirichlet_problem_negative_sobolev}. Then for any $u \in S$, integrating by parts implies that  
\begin{align*}
0 &= T(\partial_{\nu} u|_{\Gamma}) = \tilde{T}(\partial_{\nu} u) = (\partial_{\nu} u, f) \\
 &= \lim\,(\partial_{\nu} u, v_j)_{\partial M} = \lim\,\left( (\nabla u, \nabla v_j)_M + (\Delta u, v_j)_M\right) \\
 &= \lim\,(u, \partial_{\nu} v_j)_{\partial M} = (u, \partial_{\nu} v)_{\partial M}.
\end{align*}
Here we used the facts that $(-\Delta+q)v_j = (-\Delta+q)u = 0$ and that $\partial_{\nu} P f_j|_{\partial M} \to \partial_{\nu} v|_{\partial M}$ in $H^{-s-1}(\partial M)$. Since this is true for $u = P h$ for any $h \in C^{\infty}(\partial M)$ with $h|_{\Gamma} = 0$, we see that $\partial_{\nu} v|_{\partial M \setminus \overline{\Gamma}} = 0$. Now $v = P^{(-s)} f \in H^{-s+1/2}(M)$ solves 
\[
(-\Delta+q)v = 0 \text{ in $M$}, \quad v|_{\partial M \setminus \overline{\Gamma}} = \partial_{\nu} v|_{\partial M \setminus \overline{\Gamma}} = 0.
\]
The unique continuation statement in Lemma \ref{lemma_dirichlet_problem_negative_sobolev} shows that $v=0$, which implies $f=0$ and $T=0$ as required.
\end{proof}

\begin{proof}[Proof of Proposition \ref{prop_runge_approximation_cauchy}]
The proof is analogous to that of Lemma \ref{lemma_runge_approximation_sobolev} with the space $H^s(\Gamma)$ replaced by $C^{\infty}(\overline{\Gamma})$ etc. One also uses that the dual of $C^{\infty}(\partial M)$ is $\mathscr{D}'(\partial M)$ and that any element of $\mathscr{D}'(\partial M)$ is in $H^{-s}(\partial M)$ for some $s$.
\end{proof}

\bibliographystyle{alpha}

\end{document}